\documentclass[reqno,11pt,centertags]{article}
\usepackage{amsmath,amsthm,amscd,amssymb,latexsym,upref}
\date{\today}

\input epsf
\usepackage{epsfig}
\usepackage[T2A,OT1]{fontenc}
\usepackage[ot2enc]{inputenc}
\usepackage[russian,english]{babel}
\usepackage[margin=3.5 cm]{geometry}
\usepackage{epic,eepicemu}
\usepackage[all]{xy}

\newcommand{\bbA}{{\mathbb{A}}}
\newcommand{\bbD}{{\mathbb{D}}}

\newcommand{\bbR}{{\mathbb{R}}}

\newcommand{\bbZ}{{\mathbb{Z}}}
\newcommand{\bbC}{{\mathbb{C}}}

\DeclareMathAlphabet{\mathpzc}{OT1}{pzc}{m}{it}
\newcommand{\zm}{{\mathpzc{m}}}

\newcommand{\cA}{{\mathcal{A}}}
\newcommand{\cB}{{\mathcal{B}}}

\newcommand{\cE}{{\mathcal{E}}}
\newcommand{\cF}{{\mathcal{F}}}
\newcommand{\cG}{{\mathcal{G}}}

\newcommand{\cJ}{{\mathcal{J}}}
\newcommand{\cK}{{\mathcal{K}}}

\newcommand{\cO}{{\mathcal{O}}}

\newcommand{\cS}{{\mathcal{S}}}

\newcommand{\fA}{{\mathfrak{A}}}

\newcommand{\fK}{{\mathfrak{K}}}
\newcommand{\fa}{{\mathfrak{a}}}
\newcommand{\fb}{{\mathfrak{b}}}
\newcommand{\fe}{{\mathfrak{e}}}
\newcommand{\ff}{{\mathfrak{f}}}
\newcommand{\fj}{{\mathfrak{j}}}

\newcommand{\fv}{{\mathfrak{v}}}
\newcommand{\fw}{{\mathfrak{w}}}
\newcommand{\fz}{{\mathfrak{z}}}

\newcommand{\bTh}{{\mathbf{\Theta}}}
\newcommand{\bC}{{\mathbf{C}}}
\newcommand{\ba}{{\mathbf{a}}}
\newcommand{\bb}{{\mathbf{b}}}
\newcommand{\bc}{{\mathbf{c}}}
\newcommand{\bp}{{\mathbf{p}}}
\newcommand{\bo}{{\mathbf{o}}}

\newcommand{\e}{{\epsilon}}
\renewcommand{\k}{\varkappa}

\newcommand{\z}{\zeta}

\newcommand{\vp}{{\vec{p}}}
\newcommand{\vq}{{\vec{q}}}
\newcommand{\vbp}{{\vec{\mathbf{p}}}}

\newcommand{\oc}{\overset{\circ}}
\newcommand{\is}{\mathcal{IS}_E}

\renewcommand{\Im}{\text{\rm Im}\,}
\newcommand{\GMP}{\text{\rm GMP}}
\newcommand{\KS}{\text{\rm KS}}
\newcommand{\KSA}{\text{\rm KSA}}
\newcommand{\tr}{\text{\rm tr}\,}

\newcommand{\dist}{\text{\rm dist}}
\newcommand{\Res}{\text{\rm Res}\,}

\allowdisplaybreaks \numberwithin{equation}{section}
\newtheorem{theorem}{Theorem}[section]

\newtheorem{lemma}[theorem]{Lemma}
\newtheorem{proposition}[theorem]{Proposition}

\newtheorem{corollary}[theorem]{Corollary}

\theoremstyle{definition}
\newtheorem{definition}[theorem]{Definition}
\newtheorem{notation}[theorem]{Notation}

\newtheorem{remark}[theorem]{Remark}

\date{\today}
\title
{Killip-Simon problem and Jacobi flow on GMP matrices}
\author{P. Yuditskii\thanks{Supported by the Austrian Science Fund FWF, project no: P22025-N18.}}

\begin{document}

\maketitle
\begin{abstract}
One of the first and therefore most important theorems in perturbation theory claims that for an arbitrary self-adjoint operator $A$ there exists a perturbation $B$ of Hilbert-Schmidt class with arbitrary small operator norm, which destroys completely the absolutely continuous (a.c.) spectrum of the initial operator $A$ (von Neumann). However, if $A$ is the discrete free 1-D Schr\"odinger operator and $B$ is an arbitrary Jacobi matrix (of Hilbert-Schmidt class) the a.c. spectrum remains perfectly the same, that is, the interval $[-2,2]$. Moreover, Killip and Simon described explicitly  the spectral properties for such $A+B$. Jointly with Damanik they generalized this result to the case of perturbations of periodic Jacobi matrices in the non-degenerated case. Recall that the spectrum of a periodic Jacobi matrix is a system of intervals of a very specific nature. Christiansen, Simon and Zinchenko posed in a review dedicated to F. Gesztesy (2013)
 the following question: ``is there an extension of the Damanik-Killip-Simon theorem to the general finite system of intervals case?" In this paper we solve this problem completely. Our method deals with the Jacobi flow on GMP matrices. GMP matrices are probably  a new object in the spectral theory. They form a certain Generalization of matrices related to the strong Moment Problem, the latter ones are a very close relative of Jacobi and CMV matrices. The Jacobi flow on them is also a probably new member of the rich family of integrable systems. Finally, related to Jacobi matrices of Killip-Simon class, analytic vector bundles and their curvature play a certain role in our construction and, at least on the level of ideology, this role is quite essential.
\end{abstract}

\section{Introduction}

\subsection{Main result}

\noindent
\textit{(1) Von Neumann Theorem} \cite{vN} states that for an arbitrary self-adjoint operator $A$, having a nontrivial absolutely continuous (a.c.) component of the spectrum, there exists a self-adjont perturbation $\delta A$ of Hilbert-Schmidt class  such that $A+\delta A$ has a pure point spectrum. Moreover, $\delta A$ may have
 an arbitrary small operator norm.

Therefore, the following result is already quite non-trivial.

\noindent
\textit{(2) Deift-Killip Theorem} \cite{DK}. For a discrete one-dimensional Schr\"odinger operator with square summable potential, the absolutely continuous part of the spectrum is $[-2,2]$.

Thus, under special perturbations of Hilbert-Schmidt class (the square summable potential) the absolutely continuous spectrum of the free, discrete 1-D Schr\"odinger operator  is perfectly preserved. 
It is totally surprising that one can find a \textit{complete explicit characterization} of the spectral data if the perturbation is an arbitrary  Jacobi matrix of Hilbert-Schmidt class.

\noindent
\textit{(3) Killip-Simon Theorem} \cite{KS}. Let $d\sigma$ be a probability measure  on $\bbR$ with bounded but infinite support. As it is well known the orthonormal polynomials $P_n(x)$ with respect to this measure obey a three-term recurrence relation
\begin{equation*}\label{eq1}
xP_n(x)=a(n)P_{n-1}(x)+b(n)P_n(x)+a(n+1) P_{n+1}(x), \quad a(n)>0.
\end{equation*}
The following are equivalent:
\begin{itemize}
\item[(op)] $\sum_{n\ge 1}|a(n)-1|^2<\infty$ and $\sum_{n\ge 0}|b(n)|^2<\infty$.
\item[(sp)] The measure $d\sigma$ is supported on $[-2,2]\cup X$, and moreover
\begin{equation}\label{eq2}
\int_{-2}^2|\log\sigma'(x)|\sqrt{4-x^2}dx+\sum_{x_k\in X}\sqrt{x_k^2-4}^3<\infty.
\end{equation}
\end{itemize}
\begin{remark}
Of course the (op)-condition means that the Jacobi matrix
$$
J_+=\begin{bmatrix}
b(0)&a(1)& & \\
a(1)&b(1)& a(2)& \\
 &\ddots&\ddots&\ddots
\end{bmatrix}
$$
represents a Hilbert-Schmidt class perturbation of the matrix $\oc J_+$ with the constant coefficients 
$\oc a(n)=1$ and $\oc b(n)=0$. In this case we consider $J_+$ as an operator acting in the standard space of one-sided sequences $\l^2_+$. In its turn, the (sp)-condition means that the related \textit{spectral measure} $d\sigma$ has an absolutely continuous component supported on $[-2,2]$. Moreover, the spectral density $\sigma'(x)$ with respect to the Lebesgue measure satisfies an explicitly given integral condition, which in particular means that $\sigma'(x)\not=0$ a.e. on this interval. Besides that, the measure may have at most countably many mass points (the set $X$) outside of the given interval. Again, the corresponding set $X$ satisfies an explicitly given condition, which in particular means that the only possible accumulation points of this set are the endpoints $\pm 2$. Finally, note that there is no restriction on the \textit{singular component} of the measure $d\sigma$ on the interval $[-2,2]$.
\end{remark}

Later,  the authors jointly with David Damanik generalized their result on the case of perturbations of \textit{periodic}  Jacobi matrices. To state this theorem we need a couple of definitions.

We define a distance between two one-sided sequences $b=\{b(n)\}_{n\ge 0}$ and
$\tilde b=\{\tilde b(n)\}_{n\ge 0}$ from $\l_+^\infty$ by
\begin{equation}\label{eq3}
\dist^2(b,\tilde b)=\dist_{\eta}^2(b,\tilde b):=\sum_{n\ge 0}|b(n)-\tilde b(n)|^2\eta^{2n}, \quad \eta\in(0,1).
\end{equation}
The distance $\dist (J_+,\tilde J_+)$ between two Jacobi matrices is defined via the distances between the generating coefficient sequences. 

Let $J(E)$ be the isospectral set of periodic two-sided Jacobi matrices with a given spectral set $E\subset\bbR$.
 The distance between $J_+$ and $J(E)$ is defined in a standard way
\begin{equation*}\label{eq4}
\dist(J_+,J(E))=\inf\{\dist(J_+,\oc J_+):\ \oc J\in J(E)\},
\end{equation*}
where $\oc J_+$ is the restriction of a two-sided matrix $\oc J$ on the positive half-axis.

\noindent
\textit{(4) Damanik-Killip-Simon Theorem} (DKST) \cite{KSDp}. Assume that $J_+$ is a Jacobi matrix and let $d\sigma$ be the associated spectral measure.
The following are equivalent:
\begin{itemize}
\item[(opp)]  Let $S_+$ denote the shift operator in $\l^2_+$. Then
\begin{equation}\label{opp}
\sum_{n\ge 0}\dist^2((S_+^*)^n J_+S_+^n, J(E))<\infty.
\end{equation}
\item[(spp)] The measure $d\sigma$ is supported on $E\cup X$, and moreover
\begin{equation}\label{eq5}
\int_{E}|\log\sigma'(x)|\sqrt{\dist(x,\bbR\setminus E)}dx+\sum_{x_k\in X}\sqrt{\dist(x_k, E)}^3<\infty.
\end{equation}
\end{itemize}

\begin{remark}
Note that \eqref{opp}  means that the shifts of the given Jacobi matrix $J_+$ approach to the isospectral set $J(E)$, but possibly not to a specific element $\oc J$ of this set. In the same time  \eqref{eq5} looks as a  straightforward counterpart of  condition \eqref{eq2}.
\end{remark}

\begin{remark}
Let us point out that the spectral set of any periodic two-sided Jacobi matrix $\oc J$ is a system of intervals of a very special nature: the  system of intervals $E=[\bb_0,\ba_0]\setminus\cup_{j=1}^g(\ba_j,\bb_j)$ represents the spectrum of a periodic Jacobi matrix if and only if $E=T_m^{-1}([-2,2])$, where  $T_m(z)$ is a polynomial   with only real critical points, that is,
$
T_m'(c)=0 \ \text{for} \ c\in\bbR,
$
and its critical values $T_m(c)$ obey the conditions $|T_m(c)|\ge 2$. Actually, the Damanik-Killip-Simon Theorem was proved under an additional regularity condition $|T_m(c)|>2$ for all critical points $c$. In this case the degree $m=g+1$.
\end{remark}

The paper \cite{jreview} reviews recent progress in the understanding of the class of so-called  \textit{finite gap} Jacobi matrices and their perturbations. In the end of the article the authors posed the following question: ``Is there an extension of the Damanik-Killip-Simon theorem to the general finite system of intervals $E$ case?"  In the present paper \textit{we solve completely this problem}, see Theorem \ref{mainhy} below. Naturally, this question was posed as soon as  the original Killip-Simon theorem was published or even presented or proved. From this point of view \cite{jreview} is just an explicit recent reference. 

Finite gap Jacobi matrices were discovered in 
the context of approximation theory
\cite{AKH, Akh60}, \cite[Chapter X]{AKHef}. They became especially famous because of their relation with the theory of integrable systems, for  historical comments we would refer to \cite{MaT} with many references therein. But the true meaning of this class was significantly clarified recently by C. Remling: for a system of intervals $E$ the finite gap class $J(E)$ consists of all limit points of Jacobi matrices with an essential spectrum on $E$, having this $E$ as the support of their a.c. spectrum.

\noindent
\textit{(5) Remling Theorem} \cite{REMA11}. Let $E$ be a system of intervals.
Let $J_+$ be a Jacobi matrix with the generating coefficient sequences $\{a(n),b(n)\}$ such that its spectrum $\sigma(J_+)=E\cup X$, where $X$ is a set of points, which accumulate only to the endpoints of the intervals, and $\sigma'(x)\not=0$ for a.e. $x\in E$. If 
$$
\oc a(n)=\lim_{m_k\to +\infty} a(n+m_k),\quad \oc b(n)=\lim_{m_k\to +\infty} b(n+m_k),
$$
for all $n\in\bbZ$, then the corresponding two-sided Jacobi matrix $\oc J$ belongs to $J(E)$.

Note that the system of shifts $\{(S_+^*)^nJ_+S_+^n\}_{n\ge 0}$ forms a precompact set in the compact-open topology (generated by the distance \eqref{eq3}).

For $E=[\bb_0,\ba_0]\setminus\cup_{j=1}^g(\ba_j,\bb_j)$ the class $J(E)$ represents a $g$-dimensional torus, which can be parametrized explicitly.

\noindent
\textit{(6) Baker-Akhiezer parametrization for the class $J(E)$}, see e.g. \cite[Theorem 9.4]{GT}.
For $\alpha\in \bbR^g/\bbZ^g$ let 
\begin{equation}\label{aal}
\cA(\alpha)=\bar a^2\frac{\theta(\alpha+\mu+\bar\alpha)\theta(\alpha-\mu+\bar\alpha)}{\theta(\alpha+\bar\alpha)^2},
\ \ 
\cB(\alpha)=\bar b+\partial_\xi\ln\frac{\theta(\alpha-\mu+ \bar\alpha)}{\theta(\alpha+\bar\alpha)},
\end{equation}
 where 
 $$
 \theta(z)=\theta(z,\Omega)=\sum_{n\in \bbZ^g}e^{\pi i\langle \Omega n,n \rangle+2\pi i\langle z,n \rangle}, \quad z\in \bbC^g,
 $$
 with the following system of parameters depending on $E$: 
 \begin{itemize}
 \item $\Omega$ is a symmetric $g\times g$ matrix with a positive imaginary part, $\Im \Omega >0$; 
 \item $\bar\alpha\in\bbC^g$ is an appropriate shift; 
 \item $\mu\in\bbR^g/\bbZ^g$ and $\xi\in\bbR^g$ are certain fixed directions of discrete and continuous translations  on the torus $\bbR^g/\bbZ^g$, respectively; 
 \item $\bar a>0$ and $\bar b\in\bbR$ are normalization constants. 
 \end{itemize}
 Then $\oc J\in J(E)$ if and only if
 \begin{equation}\label{eqparam}
\oc a(n)^2=\cA(\alpha-\mu n), \quad
\oc b(n)=\cB(\alpha-\mu n),
\end{equation}
for some $\alpha\in \bbR^g/\bbZ^g$. In this case we write $\oc J=J(\alpha)$. Thus,
\begin{equation}\label{param1}
J(E)=\{J(\alpha):\ \alpha\in\bbR^g/\bbZ^g\}.
\end{equation}

\begin{definition}
For an arbitrary finite system of intervals $E$, we say that a Jacobi matrix $J_+$ belongs to the Killip-Simon class  $\KS(E)$ if  for some $X$ the corresponding spectral measure $d\sigma$ is supported on
$E\cup X$ and obeys \eqref{eq5}.
\end{definition}
\begin{theorem}\label{mainhy}
$J_+$ belongs to $\KS(E)$ if and only if there exist $\epsilon_\alpha(n)\in\l_+^2(\bbR^g)$ and $\epsilon_a(n)\in \l^2_+$,
$\epsilon_b(n)\in \l^2_+$ such that (cf. \eqref{eqparam})
\begin{equation}\label{132}
a(n)^2=\cA(\sum_{k=0}^n\epsilon_\alpha(k)-\mu n)+\epsilon_a(n),\ \ 
b(n)=\cB(\sum_{k=0}^n\epsilon_\alpha(k)-\mu n)+\epsilon_b(n),
\end{equation}
where $\cA(\alpha)$ and $\cB(\alpha)$ are defined in \eqref{aal}.
\end{theorem}
\begin{remark} In the one interval case the functions $\cA$ and $\cB$ are constants, e.g. if $E=[-2,2]$, then $\cA=1$ and $\cB=0$ and we obtain the original Killip-Simon Theorem.
\end{remark}

\begin{remark}\label{rem7}
It is easy to see that a Jacobi matrix of the form \eqref{132} satisfies \eqref{opp}, see Lemma \ref{lem72}. Moreover, from our explicit formulas one can give immediately a suitable approximant  for $(S_+^*)^nJ_+S^n_+$, this is
$J(\alpha_n)\in J(E)$, $\alpha_n=\sum_{k=0}^n\epsilon_\alpha(k)-\mu n$; or conclude that, if the series $\beta:=\sum_{k=0}^\infty\epsilon_\alpha(k)$ conditionally converges, then the coefficients  of $J_+$ approach, in fact, to the coefficients of the fixed element  $J(\beta)\in J(E)$,
$$
a^2(n)-\cA(\beta-\mu n)\to 0\ \text{and}\ b(n)-\cB(\beta-\mu n)\to 0,\ \text{where}\ n\to\infty.
$$
The representation \eqref{132} contains a certain ambiguity, for the reason see Remark \ref{rem71}.
\end{remark}

\subsection{Basic ideas of the method and the structure of the paper}

\noindent
\textit{The proof of DKST}   was based on two things:
\begin{itemize}
\item[(i)] Magic formula for periodic Jacobi matrices
\item[(ii)] Matrix version of the Killip-Simon theorem
\end{itemize}

The first one is the following identity. Let $S$ be the shift in the space of two sided sequences $\l^2$. If 
$E=[\bb_0,\ba_0]\setminus\cup_{j=1}^g(\ba_j,\bb_j)=T^{-1}_{g+1}([-2,2])$, then
\begin{equation*}\label{mf}
T_{g+1}(\oc J)=S^{g+1}+S^{-(g+1)}
\end{equation*}
for all $\oc J\in J(E)$. The last matrix can be understood as the $(g+1)\times(g+1)$-block Jacobi matrix with  constant block coefficients $\oc A(n)=I_{g+1}$ and $\oc B(n)=\mathbf{0}_{g+1}$.

Now, for $J_+$ the matrix $T_{g+1}(J_+)$ is a $(2g+3)$-diagonal matrix, or, also a one-sided $(g+1)\times(g+1)$ Jacobi block-matrix, see survey \cite{DPS},
$$
T_{g+1}(J_+)=\begin{bmatrix}
B(0)&A(1)& & \\
A(1)&B(1)& A(2)& \\
 &\ddots&\ddots&\ddots
\end{bmatrix}.
$$
Such matrix has a spectral $(g+1)\times(g+1)$ matrix-measure, say $d\Xi$. According to \cite{KSDp} the matrix analog of \eqref{eq2} is of the form
\begin{equation}\label{eq12}
\int_{-2}^2|\log\det\Xi'(y)|\sqrt{4-y^2}dy+\sum_{y_k\in Y}\sqrt{y_k^2-4}^3<\infty,
\end{equation}
as before  $[-2,2]\cup Y$ is the support of $d\Xi$. On the one hand this condition can be rewritten by means of the spectral measure $d\sigma$ of the initial Jacobi matrix $J_+$ into the form \eqref{eq5}, $y=T_{g+1}(x)$. On the other hand, due to the matrix version of the Killip-Simon theorem, \eqref{eq12} is equivalent to $T_{g+1}(J_+)-(S_+^{g+1}+(S_+^*)^{g+1})$ belongs to the Hilbert-Schmidt class. This is a certain bunch of conditions on the coefficients of $J_+$, but we should recognize that extracting from this simple-looking condition the final one \eqref{opp}, is a very non-trivial task.

\smallskip
Our \textit{first basic observation} is the following.

\begin{lemma}
For a system of intervals $E$  there exists a unique rational function $\Delta(z)$, $\Delta(\infty)=\infty$, such that
\begin{equation*}\label{eq13}
E=[\bb_0,\ba_0]\setminus\bigcup_{j=1}^g(\ba_j,\bb_j)=\Delta^{-1}([-2,2]),
\end{equation*}
and $\Im \Delta(z)>0$ for $\Im z>0$.
\end{lemma}
\begin{proof}
Let $\Psi(z)$ be the Ahlfors function in the domain $\bar\bbC\setminus E$. Among all analytic functions in this domain, which vanish at infinity and are bounded by one in absolute value, this function has the biggest possible value Cap$_a(E)=|z\Psi(z)|_{z=\infty}$ (the so-called analytic capacity) \cite{ALF}. As it is well known \cite{Pom0}
\begin{equation}\label{eq14}
\frac{1-\Psi(z)}{1+\Psi(z)}=\sqrt{\prod_{j=0}^g\frac{z-\ba_j}{z-\bb_j}}.
\end{equation}
Then
\begin{equation}\label{eq15}
\Delta(z)=\frac{1}{\Psi(z)}+\Psi(z)=\lambda_0z+\bc_0+\sum_{j=1}^g\frac{\lambda_j}{\bc_j-z},
\end{equation}
where $\lambda_j>0$, $j\ge 0$, and $\Psi(\bc_j)=0$, $\bc_j\in(\ba_j,\bb_j)$, $j\ge 1$.
\end{proof}

Note that in this proof we represented $\Delta(z)$ as a superposition of a function $\Psi:\bar\bbC\setminus E\to \bbD$ with the Zhukovskii map. Essentially, \eqref{eq15} is our \textit{generalized magic formula}, though it holds of course not for Jacobi matrices.

\smallskip
\noindent
\textit{Jacobi, CMV and SMP matrices}. Jacobi matrices are probably the oldest object in the spectral theory of self-adjoint operators. They are generated by the moment problem \cite{AKHmp}
\begin{equation}\label{eq16}
s_k=\int x^k d\sigma.
\end{equation}
In this problem we are looking for a measure $d\sigma$ supported on the real axis, which provides the representation 
\eqref{eq16} for the given moments $\{s_k\}_{k\ge 0}$. In this sense CMV matrices are related to the \textit{trigonometric moment problem}, which corresponds to the same question with respect to a measure supported on the unit circle. Note that this problem is also classical \cite{AKHmp}, but corresponding CMV matrices are a comparably fresh object in the spectral theory \cite{2005v1, 2005v2}. The \textit{strong moment problem} corresponds to measures on the real axis in the case that the moments are given for \textit{all integers} $k$.  An extensive bibliography of works on the strong moment problem can be found in the survey  \cite{JN}, concerning its matrix generalization see \cite{Sim1,Sim2}. 

As usual the solution of the problem deals with the orthogonalization of the generating system of functions, that is, the system
$$
1, \frac {-1} x, x, \frac {(-1)^2}{x^2}, x^2,\dots
$$
in the given case. The multiplication operator by the independent variable in $L^2_{d\sigma}$ with respect to the related \textit{orthonormal basis we call SMP matrix} (this is exactly the way of the appearance of Jacobi and CMV matrices in connection with the power and trigonometric moment problem, respectively). In another terminology they are called Laurent-Jacobi matrices \cite{BD2,DUD,HN}. Very similar to the CMV-case, this is a five-diagonal matrix of a special structure, say $A_+=A_+(d\sigma)$. We assume that the measure is compactly supported and the origin does not belong to the support of this measure. In this case our $A_+$ is bounded, moreover $A_+^{-1}$ is also a bounded operator of a similar five-diagonal structure (just shifted by one element!)

Note that, by a linear change of variable, we can always normalize an arbitrary \textit{two intervals} system to the form
$\bc_1=0$, see \eqref{eq15}, that is,
\begin{equation}\label{eq17}
E=[\bb_0,\ba_0]\setminus(\ba_1,\bb_1)=\Delta^{-1}([-2,2]),\quad \Delta(z)=\lambda_0+\bc_0-\frac{\lambda_1}{z}.
\end{equation}
Without going into detail, dealing with the structure of SMP matrices, we can formulate our \textit{second basic observation}. 
\begin{proposition}\cite{EPY}
Let $A(E)$ be the set of all two sided SMP matrices of period two with their spectrum on $E$ \eqref{eq17}.
Then $\oc A\in A(E)$ if and only if 
\begin{equation}\label{eq18}
\Delta(\oc A)=\lambda_0\oc A+\bc_0-\lambda_1 (\oc A)^{-1}=S^2+S^{-2}.
\end{equation}
\end{proposition}

\begin{remark} 
It is highly important in \eqref{eq18} to be hold  that both $\oc A$ and $(\oc A)^{-1}$ are five-diagonal matrices.
\end{remark}

Naturally, \eqref{eq17}-\eqref{eq18} have to be generalized to the multi-interval case. This leads to the concept of GMP
matrices (G for generalized), see the next subsection. However, even after such a generalization the result on spectral properties of (``some") GMP matrices of Killip-Simon class would be interesting probably only
to a small circle of specialists, working with the strong moment problem. The point is that GMP matrices are used here as a certain \textit{intermediate} (but very important) object. In a sense, this is the best possible choice of a \textit{system of coordinates}. We can try to clarify the last sentence. The standard point of view on $J(E)$ is to associate it with the hyperelliptic Riemann surface $\mathfrak{R}_E=\{(z,w):\ w^2=\prod_{j=0}^g(z-\ba_j)(z-\bb_j)\}$. Then $J(E)$ corresponds to the ``real part" of the Jacobian variety Jac$(\mathfrak{R}_E)$ of this surface see e.g. \cite{MaT, MUM}.  Periodic GMP matrices, satisfying 
\begin{equation}\label{magic}
\Delta(\oc A)=S^{g+1}+S^{-(g+1)}
\end{equation}
for $\Delta(z)$ given in \eqref{eq15}, are most likely the best possible choice for a coordinate system on the affine part of Jac$(\mathfrak{R}_E)$, at least in the application to spectral theory.

Thus, the point is to \textit{go back to Jacobi matrices}. Let $d\sigma$ be compactly supported and $0$ does not belong to its support. We can define the map 
$$
\cF_+: \text{SMP}\to \text{Jacobi}
$$ just setting $J_+(\sigma)$ in correspondence with the given $A_+(\sigma)$.
If so, we can define (in a naive way) a discrete dynamical system (\textit{Jacobi flow on SMP matrices}) by the map $\cJ_+$, which corresponds to the following commutative diagram:
\begin{equation}\label{defjfg+}
\begin{array}{ccc}
 \text{SMP} & \xrightarrow{\mathcal J_+}  & \text{SMP}  \\
   &   &    \\
_{\mathcal F_+}  \big\downarrow &   &  _{\mathcal F_+} \big\downarrow  \\
   &   &    \\
\text{Jacobi}  &  \xrightarrow{\mathcal S_+}  &  \text{Jacobi} 
\end{array}
\end{equation}
where $\cS_+ J_+=S^*_+ J_+ S_+$.

The  \textit{third basic observation} deals with the idea of getting properties of the class $\KS(E)$ from the corresponding properties of the class of SMP (or, generally, GMP) matrices using the above introduced dynamical system
$
A_+(n)=\cJ_+^{\circ n}(A_+).
$

In the next subsection we give  formal definitions for GMP matrices and the Jacobi flow on them,  but probably we can already outline the structure of the current paper:

Section 2. Probably this is a bit unusual, but we start with \textit{inverse spectral theory} for periodic GMP matrices. We recall the functional model for finite gap Jacobi matrices. In this model each operator is marked by a Hardy
space $H^2(\alpha)$ of character-automorphic functions  in the domain $\bar\bbC\setminus E$, where $\alpha$ is a character of the fundamental group of this domain \eqref{aleqal}, so, as before, $\alpha\in\bbR^g/\bbZ^g$ cf. \eqref{param1}. Here $J(\alpha)$ is the multiplication operator by the independent variable with respect to the basis $\{\fe^\alpha_n\}_{n\in\bbZ}$
\eqref{defeal}, and $\{\fe^\alpha_n\}_{n\ge 0}$ is an intrinsic basis in $H^2(\alpha)$.  The point is that  in this domain the \textit{inner} function $\Psi(z)$ and the fixed ordering $\bC=\{\bc_1,...,\bc_g\}$ of its zeros generate another natural basis $\{\ff^\alpha_n\}_{n\ge 0}$ in $H^2(\alpha)$ \eqref{gsmpbase}. Thus, we obtain a new family of operators
(the same multiplication operator in the new bases)
$$
A(E,\bC)=\{A(\alpha,\bC):\ \alpha\in\bbR^g/\bbZ^g\}.
$$
This is the collection of all periodic GMP matrices associated with the given spectral set $E$ and a fixed ordering $\bC$ of zeros of the Ahlfors function $\Psi(z)$. The fact that $\Psi(z)$ is \textit{single valued} (the character corresponding to this function is trivial) is responsible for the periodicity of an arbitrary $A(\alpha,\bC)$. 

Another characteristic feature of $\Psi(z)$ is its certain conformal invariance. Indeed, if $w=w_j=\frac{1}{c_j-z}$, then $\Psi_j(w):=\Psi(z)$ is the Ahlfors function in the $w$-plane. The given ordering $\bC$ generates the specific ordering
$$
\bC_j=\left\{\frac 1{\bc_{j+1}-\bc_j},\dots,\frac 1{\bc_{g}-\bc_j},0,\frac 1{\bc_{1}-\bc_j},\dots,\frac 1{\bc_{j-1}-\bc_j}\right\}
$$
and the multiplication by $w$  is again a \textit{periodic GMP matrix} (up to an appropriate shift). That is,
\begin{equation}\label{aut1}
S^{-j}(\bc_j-A(\alpha,\bC))^{-1}S^j\in A(E_j,\bC_j),
\end{equation}
where  $E_j=\{y=\frac 1{c_j-x}:\ x\in E\}$. Note that $0=w(\infty)$.
Let us point out that the spectral condition \eqref{eq5} possesses the same conformal invariance property.
Thus, passing from the $\fe$-basis to the $\ff$-basis in $H^2(\alpha)$, we payed a certain prize: $J(\alpha)$ is three diagonal and  $A(\alpha,\bC)$ is
a $(2g+3)$-diagonal matrix. 
 In the same time we essentially win, since $(\bc_j-J(\alpha))^{-1}$ has infinitely many non-trivial diagonals, but due to \eqref{aut1} all matrices
$(\bc_j-A(\alpha,\bC))^{-1}$ are still $(2g+3)$ diagonal.
For them \eqref{eq15} (in the chosen basis) is nothing but the magic formula \eqref{magic}.

The Jacobi flow on $A(E,\bC)$ can be defined in a very natural way. Since
$S^{-1}J(\alpha) S=J(\alpha-\mu)$ is a shift by a fixed character $\mu$, we set 
$
\cJ A(\alpha,\bC)=A(\alpha-\mu,\bC).
$
As we see, this is just one, probably new, object in the family of integrable systems.

As a result, thanks to this section we are well prepared to understand and describe the structure of GMP matrices, $A\in \GMP(\bC)$, and the Jacobi flow on them, $A(n)=\cJ^{\circ n}A$, in the general case. This  is done in the Sections 3 and 4, respectively.

 In Section 5 we work with the Killip-Simon spectral condition for two-sided Jacobi and GMP matrices.
Let us explain this passage to two-sided matrices. Our definition \eqref{defjfg+} is naive for the following reason. In the transformation $J_+(n)=\cS_+^{\circ n} J_+$ the eigenvalues in the gaps start to move. E.g., in a generic case for an initial $\oc J_+$, which corresponds to one of our fundamental  operators $\oc J\in J(E)$, the eigenvalues will cover densely  the spectral gaps $(\ba_j,\bb_j)$. Thus, corresponding to such measures $A_+(n)$ just can not be properly defined. The easiest way to explain that nevertheless our program is doable is the following:  use  two-sided Jacobi matrices and enjoy unitarity of the shift $S$ in $\l^2$!\footnote{One can  actually work with one-sided matrices but still use methods related to two dimensional cyclic subspaces.} We show that an arbitrary one-sided Jacobi matrix $J_+$, with its essential spectrum on $E$, can be extended by a Jacobi matrix $J_-=P_-\oc JP_-$, $\oc J\in J(E)$, such that for the resulting two-sided matrix $J$
\begin{equation}\label{ins17}
(\bc_j-J)^{-1}\  \text{exists for all $\bc_j$},
\end{equation}
see Lemma \ref{lemma51ini}.
We can improve the diagram \eqref{defjfg+}, see \ref{defjfg}, using the map $J=\cF A$ on two-sided matrices, see Definition \ref{def115}.  In Proposition \ref{prop73} we describe its image, which consists of Jacobi matrices with the property \eqref{ins17}.

Using the block-matrix version of the Killip-Simon theorem, it is a fairly simple task to write the necessary and sufficient condition for $A\in\GMP(\bC)$ with the spectral data \eqref{eq5}  in the form 
\begin{equation}\label{inhs}
\Delta(A)-(S^{-(g+1)}+S^{g+1})\  \text{is in the Hilbert-Schmidt class}.
\end{equation}
Note that the  relation between corresponding spectral densities of $V(A)$ and $A$ has a quite elegant form \eqref{densks}.

Section 6.
Condition  \eqref{inhs} is equivalent to
\begin{equation}\label{inhs2}
H_\pm(A)<\infty
\end{equation}
for the \textit{Killip-Simon functional of the problem}, which is basically the $\l^2_{\pm}$-part of the trace of $(\Delta(A)-(S^{-(g+1)}+S^{g+1}))^2$, for the precise expression see \eqref{eq10}. In the spirit of our third basic observation, we compute the ``derivative" of this functional in the direction of the Jacobi flow, that is, the value
$$
\delta_{\cJ}H_+(A):=H_+(A)-H_+(\cJ A),
$$ 
see Lemma \ref{lemjder}. \textit{This derivative represents a finite sum of squares!} Now, we can rewrite \eqref{inhs2} as the ``integral" $\sum_{n\ge 0}\delta_{\cJ}H_+(\cJ^{\circ n}A)<\infty$ and, thus,  get  certain $\l^2$-properties. Note that they are already more related to the Jacobi matrix $J=\cF A$,  than to the given GMP matrix $A$ itself.
Nevertheless, all these conditions were given by means of the coefficients of $\Delta(A)$, not by the ones of $A$ (or the system of iterates $A(n)$, to be more precise). This is probably the hardest technical part of the work. To indicate the difficulty, we would 
like to mention the following. In \cite{NPVY} we found higher-order generalizations of Killip-Simon sum rules (relations between coefficients of $J_+$ and the spectral measure $d\sigma$), for a \textit{single interval spectrum}. But only for a very special family (related to Chebyshev polynomials of an arbitrary degree $n$), which was initially found in \cite{LNS}, we were able to convert the result of the form \eqref{inhs} to explicit relations on the coefficients of the given $J_+$. Otherwise, each particular case becomes a reason for an interesting research, see e.g. \cite{K2004, GZ,SZ}. Moreover, a nice looking general conjecture was recently disproved by M. Lukic \cite{LU}. By the  way, for a highly interesting new development in this area see \cite{GNR}. 
So, in  this section we prove Theorem \ref{th73}; practically, this is already a parametric representation for coefficients of Jacobi matrices of $\KS(E)$. 

Section 7. In this section we finalize the parametric representation for Killip-Simon Jacobi matrices associated to an arbitrary system of intervals $E$, that is, we prove the main Theorem \ref{mainhy}. In the end of this section we demonstrate implicitly our \textit{last basic for this paper observation} that the spectral theory in the spirit of \cite{CD} could be more powerful than the  classical orthogonal polynomials approach \cite{AKHmp, BS}, see Subsection \ref{subs72} and especially Remark \ref{inirem74}. Explicitly this was demonstrated in \cite{PY, VY, PVY}, as well as in Section 2 of the current paper.
At the moment we are not able to present a theory of spaces of vector bundles, which corresponds as  model spaces to Jacobi matrices of Killip-Simon class (in full generality) even in a finite gap case. 

Basic facts with respect to one sided GMP matrices are given in the appendix.

\subsection{GMP matrices and Jacobi flow on them in solving the Killip-Simon problem}

In this subsection we give formal definitions for the named objects so that in the end of it we are able to state Theorem \ref{th73}. This is the main ingredient in our proof of Theorem \ref{mainhy}.

\begin{itemize}
\item
Let $\{e_n\}$ be the standard basis in $\l^2$. Depending on the context, $\l^2_+$ is the  set of square-summable one-sided sequences or the subspace of $\l^2$ spanned by $\{e_n\}_{n\ge 0}$. In the last case $\l^2_-:=\l^2\ominus \l^2_+$ and 
$P_\pm:\l^2\to \l^2_\pm$ are the orthogonal projectors. 

\item
Let $\{\delta_k\}_{k=0}^g$ denote the standard basis in the Euclidian space $\bbC^{g+1}$.

\item
By $T^*$ we denote the conjugated operator to an operator  $T$, or the conjugated matrix  if $T$ is a matrix.
In particular, for a vector-column $\vp\in \bbC^{g+1}$, $(\vp)^*$ is a $(g+1)$-dimensional vector-row. Consequently, the scalar product in $\bbC^{g+1}$ can be given in the following form
$
\langle \vp,\vq \rangle=(\vq)^*\vp.
$

\item
 The notation $T^-$ denotes the upper triangular part of a matrix $T$ (\textit{excluding the main diagonal}), respectively 
$T^+:=T-T^-$ is its lower triangular part (\textit{including the main diagonal}).
\end{itemize}
GMP matrices form a certain special subclass of real symmetric $(2g+3)$-diagonal matrices, $g\ge 1$.
First of all, the class depends on an ordered collection of distinct points $\bC=\{\bc_1,\dots,\bc_g\}$. That is, if needed we will specify the notation $\GMP(\bC)$. We will define two-sided GMP matrices, but their restrictions on the positive half-axis will be highly important.

\begin{definition}
We say that $A$ is of the class $\bbA$ if it is a $(g+1)$-block Jacobi matrix
\begin{equation}\label{n1}
A=\begin{bmatrix}
\ddots&\ddots&\ddots&& &\\
&A^*(\vp_{-1})&B(\vbp_{-1})&A(\vp_0)& & \\
& &A^*(\vp_{0})&B(\vbp_{0})&A(\vp_1)& \\
& & &\ddots&\ddots&\ddots
\end{bmatrix}
\end{equation}
such that
\begin{equation}\label{n2}
\vbp=(\vp,\vq)\in\bbR^{2g+2},\quad A(\vp)=\delta_g \vp\,^*,
\quad
B(\vbp)
=(\vq \vp\,^*)^-+(\vp\vq\,^*)^++\tilde\bC,
\end{equation}
and
\begin{equation}\label{n3}
\tilde \bC=\begin{bmatrix}
\bc_1& & & \\
& \ddots& & \\
& & \bc_g & \\
& & &0
\end{bmatrix},\ 
\vp_j=
\begin{bmatrix}
p^{(j)}_0\\
\vdots\\
p^{(j)}_g
\end{bmatrix}, \ 
\vq_j=
\begin{bmatrix}
q^{(j)}_0\\
\vdots\\
q^{(j)}_g
\end{bmatrix}, \quad p^{(j)}_g>0.
\end{equation}
 We call $\{\vbp_j\}_{j\in\bbZ}$ the generating coefficient sequences (for the given $A$).
\end{definition}

\begin{definition} Let $S$ be the shift operator $Se_n=e_{n+1}$.
A  matrix $A\in\bbA$ belongs to the GMP class if the matrices $\{\bc_k-A\}_{k=1}^g$ are invertible, and moreover
$S^{-k}(\bc_k-A)^{-1}S^k$ are also of the class $\bbA$, see \eqref{n1}-\eqref{n3}. To abbreviate we write $A\in \GMP(\bC)$.
\end{definition} 

\begin{remark}
As it follows from the definition  $\|S^{-k}(\bc_k-A)^{-1}S^k\|<\infty$. These  conditions can be written explicitly as a certain set of conditions on  the coefficients $\{\vbp_j\}_{j\in\bbZ}$  of the initial matrix $A\in\bbA$, see \eqref{altdef}. That is, in fact, $A\in\GMP(\bC)$ if and only if it is of the class $\bbA$ for a certain ordered collection $\{\bc_1,\dots,\bc_g\}$ and \eqref{altdef} holds for the generating sequences.
 This can be regarded as a \textit{constructive definition}  of GMP matrices, see 
Theorem \ref{defaltdef}. 
\end{remark}

Let $J$ be a Jacobi matrix with coefficients
$\{a(n),b(n)\}$:
\begin{equation}\label{ijf1}
Je_n=a(n) e_{n-1}+b(n) e_n +a(n+1) e_{n+1}, \quad a(n)>0, \ n\in\bbZ.
\end{equation}
The two-dimensional space spanned by $e_{-1}$ and $e_0$ forms a cyclic subspace for $J$. Also, $J$ can be represented as a two-dimensional perturbation of the orthogonal sum with respect to the decomposition $\l^2=\l^2_-\oplus \l^2_+$
\begin{equation}\label{ijf2}
J=\begin{bmatrix}
J_-& 0\\
0& J_+
\end{bmatrix}+a(0)(e_0\langle \cdot, e_{-1} \rangle+e_{-1}\langle \cdot, e_{0} \rangle ).
\end{equation}
We have a similar decomposition for $A\in\GMP(\bC)$
\begin{equation}\label{ijf3}
A=\begin{bmatrix}
A_-& 0\\
0& A_+
\end{bmatrix}+\|\vp_0\|(\tilde e_0\langle \cdot, \tilde e_{-1} \rangle+\tilde e_{-1}\langle \cdot, \tilde e_{0} \rangle ),
\ \tilde e_{-1}=e_{-1}, \ \tilde e_0:=\frac{1}{\|\vp_0\|}P_+ A e_{-1}. 
\end{equation}

\begin{definition}\label{def115}
For $A\in \GMP(\bC)$ the Jacobi matrix $J=\cF A$ is uniquely defined by the conditions
\begin{equation}\label{ijf4}
r_{\pm}(z):=\langle (J_{\pm}-z)^{-1} e_{\frac{-1\pm 1}{2}}, e_{\frac{-1\pm 1}{2}}\rangle=
\langle (A_{\pm}-z)^{-1}\tilde e_{\frac{-1\pm 1}{2}}, \tilde e_{\frac{-1\pm 1}{2}}\rangle, \quad
a(0)=\|\vp_0\|.
\end{equation}
\end{definition}

\begin{definition} Let
$\cS J:= S^{-1}J S$. 
The Jacobi flow on GMP matrices is generated by the transformation $\cJ$, which makes the following diagram commutative
\begin{equation}\label{defjfg}
\begin{array}{ccc}
 \text{GMP} & \xrightarrow{\mathcal J}  & \text{GMP}  \\
   &   &    \\
_\mathcal F  \big\downarrow &   &  _\mathcal F \big\downarrow  \\
   &   &    \\
\text{Jacobi}  &  \xrightarrow{\mathcal S}  &  \text{Jacobi} 
\end{array}
\end{equation}
The corresponding discrete dynamical system (Jacobi flow) is of the form 
$$
A(n+1)=\cJ A(n), \quad A(0)=A.
$$
\end{definition}

Essentially, it can be reduced to an open (input-output) dynamical system \eqref{ods14}. The coefficients of the Jacobi matrix $J=\cF A$ are easily represented by means of the Jacobi flow acting on the initial $A$. Namely,

\begin{corollary}\label{cor113}
Let $J=\cF A$ and $A(n)=\cJ^{\circ n} A$. In the above notations \eqref{ijf1}
\begin{equation}\label{coefflow}
a(n)=\|\vp_0(n)\|, \quad b(n-1)=q_g^{(-1)}(n)p_g^{(-1)}(n).
\end{equation}
\end{corollary}

Now we can define the Killip-Simon class of GMP matrices. Let $E$ be a system of $g+1$ disjoint intervals, $E=[\bb_0,\ba_0]\setminus \cup_{j=1}^g(\ba_j,\bb_j)$.
Let $\Delta(z)=\Delta_E(z)$ be the unique function, which was given in \eqref{eq15}.

\begin{proposition}\label{prop18}
 $\oc A\in\GMP(\bC)$, generated by coefficients $\vbp=(\vp,\vq)$,  belongs to the isospectral set  of periodic matrices $A(E,\bC)$ if and only if 
it obeys the magic formula \eqref{magic}. Moreover, the 
isospectral surface $\is$ is given by, see \eqref{explpqg}, \eqref{explrhk} and \eqref{defla},
\begin{equation}\label{iso101}
p_g=\frac 1{\lambda_0}, \quad q_g=-\bc_0-\lambda_0\sum_{j=1}^{g-1} p_jq_j, \quad \Lambda_k(\vbp)=\lambda_k, \ k=1,...,g.
\end{equation}
\end{proposition}

\begin{definition}\label{defini1} Let $A\in\GMP(\bC)$. Let $\sigma_\pm$ be the related spectral measures, that is,
\begin{equation*}\label{specmes}
r_{\pm}(z)=\int \frac{d\sigma_{\pm}(x)}{x-z},
\end{equation*}
where $r_{\pm}(z)$ are given in \eqref{ijf4}.
We say that $A$ belongs to the Killip-Simon class
$\KSA(E, \bC)$ if
the measures $\sigma_\pm$ are supported on $E\cup X_{\pm}$, and both satisfy \eqref{eq5}.
\end{definition}

The following theorem is essentially a consequence of the matrix version of the Killip-Simon theorem.

\begin{theorem}
$A\in\GMP(\bC)$ belongs to the Killip-Simon class $\KSA(E,\bC)$ if and only if
the difference $\Delta_E(A)-(S^{-(g+1)}+S^{g+1})$ belongs to the Hilbert-Schmidt class.
\end{theorem}

However, the next statement is already highly non-trivial. Practically, it gives a parametrization of the coefficients of \textit{Jacobi matrices} of  Killip-Simon class with its essential spectrum on $E$.
\begin{theorem}\label{th73} 
For $A\in \GMP(\bC)$, let $A(n+1)=\cJ A(n)$, $A(0)=A$. Let $\{\vbp_j(n)\}_{j\in \bbZ}$ be the forming $A(n)$ coefficient sequences. $A$ belongs to $\KSA(E,\bC)$ if and only if 
\begin{align}
\{p^{(\pm 1)}_j(n)-p^{(0)}_j(n)\}_{n\ge 0}\in \l^2_+,\quad
 &\{q^{(\pm 1)}_j(n)-q^{(0)}_j(n)\}_{n\ge 0}\in
  \l^2_+,
\label{m29}\\
\{\lambda_0p_g^{(0)}(n)-1\}_{n\ge0}\in\l_+^2, \quad
&\{\lambda_0\langle \vp_0(n),\vq_0(n) \rangle+\bc_0\}_{n\ge 0}\in \l_+^2,
\label{m30}\\
&\{\Lambda_k (\vbp_0(n))-\lambda_k\}_{n\ge 0}\in\l^2_+
\label{m31} 
\end{align}
hold for all $j=0,\dots, g-1$ and all $k=1,\dots, g$ (cf. \eqref{m30}-\eqref{m31} and \eqref{iso101}).
\end{theorem}

 To summarize, in this paper solving the Killip-Simon problem
 \begin{itemize}
 \item we introduce GMP matrices as  possibly the best coordinate system on the Jacobians of hyperelliptic Riemann surfaces associated to finite band operators;
 \item we introduce and study the Jacobi flow on GMP matrices as  one more important object in a rich family of integrable systems;
 \item our study is based essentially on the Damanik-Killip-Simon theorem on Hilbert-Schmidt perturbations of Jacobi block-matrices with constant coefficients;
 \item we follow  the ideology of application of analytic vector bundles in spectral theory, explicitly in Section \ref{Section2} and implicitly in Section \ref{Section7}.
 \end{itemize}

\section{Functional models for $J(E)$ and $A(E,\bC)$.  Jacobi flow on periodic GMP matrices}\label{Section2}

\subsection{ Hardy spaces and the class $J(E)$}

In what follows, we will use functional models for the class of reflectionless matrices $J(E)$ in the form as considered in \cite{SY}. To this end, we need
to recall certain special functions related to function theory in
 the common resolvent domain $\Omega=\overline{\mathbb C}\setminus E$ for $J\in J(E)$. Note that in this case, $E$ can be a set of an essentially more complicated structure \cite{Has, Pom, WID}, than
 a system of intervals.  

Let $\mathbb D/\Gamma\simeq\overline{\mathbb C}\setminus E$ be a uniformization of the domain $\Omega$. It means that there exists a Fuchsian
group $\Gamma$ and a meromorphic function $\fz:\mathbb D\rightarrow\overline{\mathbb C}\setminus E$, $\fz\circ\gamma=\fz$ for all
$\gamma\in\Gamma$, such that
\begin{equation*}
 \forall z\in\overline{\mathbb C}\setminus E~\exists\zeta\in\mathbb D\!:~\fz(\zeta)=z \text{ and } \fz(\zeta_1)=\fz(\zeta_2)\Rightarrow
\zeta_1=\gamma(\zeta_2).
\end{equation*}
We assume that $\fz$ meets the normalization $\fz(0)=\infty$, $(\zeta \fz)(0)>0$.

Let $\Gamma^{*}$ be the group of characters of the discrete group $\Gamma$,
$$
\Gamma^*=\{\alpha|\ \alpha:\Gamma\to \bbR/\bbZ\ \text{such that}\ \alpha(\gamma_1\gamma_2)=\alpha(\gamma_1)+\alpha(\gamma_2)\}
$$
Since $\Gamma$ is formed by $g$ independent generators, say $\{\oc \gamma_j\}_{j=1}^g$, the group $\Gamma^*$ is equivalent to $\bbR^g/\bbZ^g$,
\begin{equation}\label{aleqal}
\alpha\simeq\{\alpha(\oc \gamma_1),\dots, \alpha(\oc\gamma_g)\}\in\bbR^g/\bbZ^g.
\end{equation}

\begin{definition}
 \label{def:htwo}
For $\alpha\in\Gamma^*$
we define the Hardy  space of character automorphic functions as
\begin{equation*}
H^2(\alpha) = H^2_{\Omega}(\alpha) = \{ f \in H^2\!:~ f \circ \gamma = e^{2\pi i\alpha(\gamma)} f,~\gamma\in\Gamma \},
\end{equation*}
where $H^2$ denotes the standard Hardy class in $\mathbb D$.
\end{definition}

Fix $z_0\in\Omega$ and let $\text{\rm orb}(\zeta_0)=\fz^{-1}(z_0)=\{\gamma(\zeta_0)\}_{\gamma\in\Gamma}$. The Blaschke product $\fb_{z_0}$ with zeros at
$\fz^{-1}(z_0)$ is called the Green function of the group $\Gamma$ (cf.~\cite{SY}). It is related to the standard Green
function $G(z,z_0)$ in the domain $\Omega$ by
$$
\log \frac{1}{|\fb_{z_0}(\zeta)|} = G\left(\fz(\zeta),z_0\right).
$$
The function $\fb_{z_0}$ is character automorphic, that is, $\fb_{z_0}\circ\gamma=e^{2\pi i\mu_{z_0}}\fb_{z_0}$, where $\mu_{z_0}\in\Gamma^{*}$. For $\fb_{z_0}$ we fix the normalization $\fb_{z_0}(0)>0$ if $z_0\not=\infty$ and $(\fz \fb)(0)>0$
for the Blaschke product $\fb$ related to infinity.

We define $k_{\zeta_0}^{\alpha}(\zeta)=k^{\alpha}(\zeta,\zeta_0)$ as the reproducing kernel of the space $H^2(\alpha)$, that is,
\begin{equation*}
 \left\langle f, k_{\zeta_0}^{\alpha}\right\rangle = f(\zeta_0)\quad \forall f\in H^2(\alpha).
\end{equation*}
\begin{remark}\label{rema22}
Let us point out that in our case this reproducing kernels possess a representation by means of $\theta$ functions associated with the given Riemann surface \cite{Fay}. As already mentioned, $k^\alpha$ has sense in a much more general situation, say, domains of Widom type. Although, generally speaking,  they can not be represented via $\theta$ functions, they still play a role of special functions in the related problems.
\end{remark}

Let $k^{\alpha}(\zeta)=k_{0}^{\alpha}(\zeta)$, $\fb(\zeta)=\fb_{\fz(0)}(z)$, and  $\mu=\mu_{\fz(0)}$. We have an evident decomposition
\begin{equation}\label{ort1}
 H^{2}(\alpha)=\{\fe^{\alpha}\}\oplus \fb H^2(\alpha-\mu), \quad \fe^{\alpha}=\frac{k^{\alpha}(\zeta)}{\sqrt{k^{\alpha}(0)}}.
\end{equation}
This decomposition plays an essential role in the proof of the following theorem.

\begin{theorem}
 \label{thm:onb}
The system of functions 
\begin{equation}\label{defeal}
 \fe_{n}^{\alpha}(\zeta)=\fb^{n}(\zeta)\frac{k^{\alpha-n\mu}(\zeta)}{\sqrt{k^{\alpha -n\mu}(0)}}
\end{equation}

\begin{itemize}
 \item[(i)] forms an orthonormal basis in $H^2(\alpha)$ for $n\in \mathbb N$ and
 \item[(ii)] forms an orthonormal basis in $L^2(\alpha)$ for $n\in\mathbb Z$,
\end{itemize}
where 
\begin{equation*}
L^2(\alpha) =  \{ f \in L^2\!:~ f \circ \gamma = e^{2\pi i\alpha(\gamma)} f,~\gamma\in\Gamma \}.
\end{equation*}
\end{theorem}
\begin{proof}
 Item (i) we obtain by iterating   \eqref{ort1}. A proof for (ii) in a much more general case can be found in \cite[Theorem E]{SY}.
\end{proof}

The following theorem describes all elements of $J(E)$ for a given finite-gap set $E$.

\begin{theorem}
 \label{thm:multbyz}
The multiplication operator by $\fz$ in $L^2(\alpha)$ with respect to the basis $\{\fe_n^{\alpha}\}$ from Theorem~\ref{thm:onb}
is the following Jacobi matrix $J=J(\alpha)$:
\begin{equation*}
 \fz \fe_{n}^{\alpha}=a(n;\alpha)\fe_{n-1}^{\alpha} + b(n;\alpha)\fe_{n}^{\alpha}+a(n+1;\alpha)\fe^{\alpha}_{n+1},
\end{equation*}
where
\begin{equation*}
 a(n;\alpha)=\mathcal A(\alpha-n\mu), \quad\mathcal A(\alpha)=(\fz \fb)(0)\sqrt{\frac{k^{\alpha}(0)}{k^{\alpha+\mu}(0)}}
\end{equation*}
and 
\begin{equation*}
 b(n;\alpha)=\mathcal B(\alpha-n \mu),~~ \mathcal B(\alpha)=\frac{\fz \fb(0)}{\fb^{\prime}(0)}+
\left\{  \frac{\left(k^{\alpha}\right)^{\prime}(0)}{k^{\alpha}(0)}- \frac{\left(k^{\alpha+\mu}\right)^{\prime}(0)}{k^{\alpha+\mu}(0)}\right\}
+\frac{\left(\fz \fb\right)^{\prime}(0)}{\fb^{\prime}(0)}.
\end{equation*}
This Jacobi matrix $J(\alpha)$ belongs to $J(E)$. Thus, we have a map from $\Gamma^{*}$ to $J(E)$. Moreover, this map is one-to-one.
\end{theorem}

\begin{remark}
Using the representation of the reproducing kernels via $\theta$ functions, see Remark \ref{rema22},
 one gets $\cA(\alpha)$ and $\cB(\alpha)$ in  the form
\eqref{aal}.
\end{remark}

\begin{remark}
The following important relation is
an immediate  consequence of the above functional model 
\begin{equation}\label{sjm}
S^{-1}J(\alpha) S=J(\alpha-\mu), \quad S e_n:= e_{n+1}.
\end{equation}
In particular, $J(\alpha)$ is periodic if and only if $N\mu=\mathbf{0}_{\Gamma^*}$ for a certain  integer $N$.
\end{remark}

\subsection{Class $A(E,\bC)$ and Jacobi flow}

Now we turn to the functional model for $A(E,\bC)$.
The rational function $\Delta(z)$ and the  single valued function $\Psi(z)$,  $z\in \bar\bbC\setminus E$, were defined in
\eqref{eq14}-\eqref{eq15}.
Let us list characteristic properties of $\Psi(z)$:
\begin{itemize}
 \item[(i)] $|\Psi|<1$ in $\Omega$ and $|\Psi|=1$ on $E$,
 \item[(ii)] $\Psi(\infty)=\Psi(\bc_j)=0$, $1\le j\le g$, otherwise $\Psi(z)\not=0$.
\end{itemize}
All this implies that 
$$
\log\frac 1{|\Psi(z)|}=G(z)+\sum_{j=1}^g G(z,\bc_j).
$$ 
Therefore $\Psi(\fz(\zeta))=\fb(\zeta)\prod_{j=1}^g \fb_{\bc_j}(\zeta)$. In particular, $\mu+\sum^g_{j=1}\mu_{\bc_j}=\mathbf{0}_{\Gamma^*}$.

Let us fix $\zeta_{j}\in\bbD$ such that $\fz(\zeta_{j})=\bc_j$ and $\oc\gamma_j(\zeta_{j})=\bar \zeta_{j}$ for the generator $\oc\gamma_j$ of the group $\Gamma$. In order to construct a functional model for operators from $A(E,\bC)$, we start with the following counterpart of the orthogonal decomposition \eqref{ort1}:
\begin{equation}\label{smpbase0}
H^2(\alpha)=\{k^{\alpha}_{\zeta_{1}},\dots, k^{\alpha}_{\zeta_{g}}, k^{\alpha}\}\oplus\Psi H^2(\alpha)
=\{\ff^{\alpha}_0\}\oplus\dots\oplus\{\ff^{\alpha}_g\}\oplus\Psi H^2(\alpha),
\end{equation}
where
\begin{equation}\label{smpbase}
\ff_0^{\alpha}=\frac{e^{-\pi i\alpha(\oc\gamma_1)}k^\alpha_{\zeta_{1}}}{\sqrt{k^\alpha_{\zeta_{1}}(\zeta_{1})}}, \
\ff_1^\alpha=\frac{e^{-\pi i(\alpha-\mu_{\bc_1})(\oc\gamma_2)}\fb_{\bc_1} k_{\z_{2}}^{\alpha-\mu_{\bc_1}}}{\sqrt{k_{\z_{2}}^{\alpha-\mu_{\bc_1}}(\z_{2})}},...,
\ \ff_g^\alpha=\frac{\prod_{j=1}^g \fb_{\bc_j} k^{\alpha+\mu}}{\sqrt{k^{\alpha+\mu}(0)}}.
\end{equation}

\begin{theorem}
 \label{thm:onbsmp}
The system of functions 
\begin{equation}\label{gsmpbase}
\ff_{n}^{\alpha}=\ff_{n}^{\alpha}(\z;\bc_1,\dots,\bc_g)=\Psi^m \ff_j^\alpha,\quad 
 n=(g+1)m+j, \ j\in[0,\dots,g]
\end{equation}
\begin{itemize}
 \item[(i)] forms an orthonormal basis in $H^2(\alpha)$ for $n\in \mathbb N$ and
 \item[(ii)] forms an orthonormal basis in $L^2(\alpha)$ for $n\in\mathbb Z$.
\end{itemize}
\end{theorem}
\begin{proof}
 Item (i) follows from \eqref{smpbase0} and for (ii) we have to use the description of the orthogonal complement 
 $L^2(\alpha)\ominus H^2(\alpha)$, see \cite{SY}.
\end{proof}

Similarly as we had before, this allows us to parametrize all elements of $A(E,\bC)$ for a given  $E$ by the 
characters of $\Gamma^*$.

\begin{theorem}
 \label{thm:multbyzsmp}
In the above notations
the multiplication operator by $\fz$ with respect to the basis $\{\ff_n^{\alpha}\}$ is a GMP matrix $A(\alpha;\bC)\in A(E,\bC)$.
Moreover, this map $\Gamma^*\to A(E;\bC)$ is one-to-one up to the identification 
$
(p_j,q_j)\mapsto (-p_j,-q_j) \ \text{in} \ A(E;\bC)$, $ 0\le j\le g-1.$
\end{theorem}

\begin{proof}
We claim that the structure of the matrix is fixed by the choice of the orthonormal basis. In particular we need to check that, under the normalization \eqref{smpbase}, $p_j(\alpha)$
 and $q_j(\alpha)$ are real. 
 
 For $\beta_j=\alpha-\sum_{k=1}^{j}\mu_{c_k}$, we have
 \begin{equation}\label{29apr1}
 p_j(\alpha)=\langle \fz \ff_j^\alpha, \ff_{-1}^\alpha \rangle=(\fb\fz)(0)\prod_{k=1}^{j-1}\fb_{\bc_k}(0)
 e^{-\pi i\beta_j(\oc\gamma_j)}\frac{k^{\beta_j}(0,\zeta_{j})}{\sqrt{k^{\beta_j}_{\zeta_{j}}(\zeta_{j})k^{\alpha+\mu}(0)}}.
\end{equation}
 Since $\overline{k^\beta(\bar \zeta)}=k^\beta(\zeta)$  for all $\beta\in\Gamma^*$, we get
 $$
k^\beta(\zeta_{j})= \overline{k^\beta(\bar\z_{j})}= \overline{k^\beta(\oc\gamma_j(\zeta_{j}))}= 
 e^{-2\pi i\beta(\oc\gamma_j)}\overline{k^\beta(\zeta_{j})}.
 $$
 Therefore, $e^{-\pi i\beta(\oc\gamma_j)}\overline{k^\beta(\zeta_{j})}=
 e^{-\pi i\beta(\oc\gamma_j)}k^\beta(0,\z_{j})$ is real. Note that the square root of $e^{-2\pi i\beta(\oc\gamma_j)}$ is defined up to the multiplier $\pm 1$.
 
 To compute the entries of the matrix $B(\vbp)$, and actually to show its specific structure,  we use a standard trick related to reproducing kernels. Let $0\le m <n\le g$. Then
 $$
 \langle \fz \ff^\alpha_n, \ff^\alpha_m \rangle=
 \langle \fz \fb_{\bc_{m}}\dots \fb_{\bc_{n-1}}
 \frac{e^{-i\pi(\beta_n(\oc\gamma_n)-\beta_m(\oc\gamma_m))}k_{\z_{n}}^{\beta_n}}{\sqrt{k^{\beta_{n}}_{\zeta_{n}}(\z_{n})}}, 
  \frac{k_{\z_{m}}^{\beta_m}}{\sqrt{k^{\beta_{m}}_{\zeta_{m}}(\z_{m})}} \rangle.
 $$
 Denoting for a moment the first function in the scalar product by $\tilde \ff$, since $\tilde \ff(\z_{m})=0$, we can continue with
 $$
 = \langle \fz \tilde \ff -(\fb\fz)(0)\tilde \ff(0)\frac{k^{\beta_m+\mu}}{\fb k^{\beta_m+\mu}(0)}
 , 
 \frac {k_{\z_{m}}^{\beta_m}}{\sqrt{k^{\beta_{m}}_{\zeta_{m}}(\z_{m})}}\rangle
  =-\frac{(\fb\fz)(0)\tilde \ff(0)}{  \sqrt{k^{\beta_{m}}_{\zeta_{m}}(\z_{m})}}\frac{k^{\beta_m+\mu}(\z_{m})}{\fb(\z_{m}) k^{\beta_m+\mu}(0)
}
 $$
 $$
 =- \frac{(\fb\fz)(0)
 (\fb_{\bc_{m}}\dots \fb_{\bc_{n-1}})(0)
e^{-i\pi\beta_n(\oc\gamma_n)}k_{\z_{n}}^{\beta_n}(0)}{\sqrt{k^{\beta_{n}}_{\zeta_{n}}(\z_{n})}}
 \frac{e^{i\pi\beta_m(\oc\gamma_m)}k^{\beta_m+\mu}(\z_{m})}{\fb(\z_{m}) k^{\beta_m+\mu}(0)\sqrt{k^{\beta_{m}}_{\zeta_{m}}(\z_{m})}}=p_n(\alpha)q_m(\alpha)
 $$
 where, as before,
 $p_n(\alpha)$ is of the form \eqref{29apr1} and
 $$
 q_m(\alpha):=-\frac{e^{i\pi\beta_m(\oc\gamma_m)}k^{\beta_m+\mu}(\z_{m})\sqrt{k^{\alpha+\mu}(0)}}
 {(\fb_{\bc_{1}}\dots \fb_{\bc_{m-1}})(0)
 \fb(\z_{m}) k^{\beta_m+\mu}(0)\sqrt{k^{\beta_{m}}_{\zeta_{m}}(\z_{m})}}.
 $$
 Similarly, we get the representation for the diagonal terms.
  
 If a periodic $\oc A\in A(E,\bC)$ is given, we introduce its resolvent function $r_+(z)$. As usual it can be expressed by means of the \textit{transfer matrix}, see Theorem \ref{th213} below. After that we define $\alpha$ exactly as in the Jacobi matrix case, see e.g. \cite[(2.3.2)-(2.3.3) and Theorem A]{SY}.
\end{proof}

\begin{definition}\label{def:312}
We define the Jacobi flow  on $A(E;\bC)$ as the dynamical system generated by the following map (see \eqref{sjm}, \eqref{defjfg}): 
$$
\mathcal J A(\alpha)=A(\alpha-\mu), \quad \alpha\in\Gamma^*.
$$
\end{definition}

We can describe this operation in a very explicit form.
\begin{lemma}
 \label{thm:jacobiflowper}
 Let
 \begin{equation}\label{oofphi}
\bo(\phi)=\begin{bmatrix}
\sin\phi&
\cos\phi\\
\cos\phi&-\sin\phi
\end{bmatrix}.
\end{equation}
Let $\cO(\alpha)$ be the unitary, periodic $(g+1)\times (g+1)$-block diagonal  matrix given by
\begin{equation}
\label{eqn:ufunctional}
 \cO(\alpha)\begin{bmatrix}e_{(g+1)m}&\dots& e_{(g+1)m+g}
 \end{bmatrix}=\begin{bmatrix}e_{(g+1)m}&\dots& e_{(g+1)m+g}
 \end{bmatrix} O(\alpha), 
 \end{equation}
 where
  \begin{equation}\label{jfp0}
O(\alpha)=
 \begin{bmatrix}
I_{g-2}&0\\
0&\bo(\phi(\alpha))
\end{bmatrix},\ 
\begin{bmatrix}
\sin (\phi(\alpha))& \cos(\phi(\alpha))
\end{bmatrix}=
\frac{\begin{bmatrix} 
p_{g-1}(\alpha)& p_g(\alpha)
\end{bmatrix}}
{\sqrt{p_{g-1}^2(\alpha)+p_g^2(\alpha)}}.
 \end{equation}
 Then
\begin{equation}\label{jfp}
 \mathcal O A(\alpha;\bC):=S^{-1}\mathcal O(\alpha)^*A(\alpha;\bC)\mathcal O(\alpha)S=
 A(\alpha+\mu_{\bc_g};\bc_g,\bc_1,\dots,\bc_{g-1}).
\end{equation}
\end{lemma}

\begin{proof}
In fact, we consider here the following elementary operation on the set of periodic GMP matrices: we switch the oder in the orthogonalization procedure of the family of reproducing kernels from $\{k^{\alpha}_{\zeta_{1}},\dots, k^{\alpha}_{\zeta_{g}}, k^{\alpha}\}$ to $\{k^{\alpha}_{\zeta_{1}},\dots, k^{\alpha}, k^{\alpha}_{\zeta_{g}}\}$. Let $\{\tilde \ff_n^\alpha\}$ be the new orthonormal system
in $L^2(\alpha)$.
It is easy to observe that
$$
\fb_{\bc_g}(\z)\tilde\ff^\alpha_n(\z)=\ff^{\alpha+\mu_{\bc_g}}_{n-1}(\z;\bc_g,\bc_1,\dots,\bc_{g-1}).
$$
 That is, up to the shift, we derived a GMP basis of the form \eqref{smpbase}, but with the new ordering $(\bc_g,\bc_1,\dots,\bc_{g-1})$ and the new character $\alpha+\mu_{c_g}$. Note that passing from one to another basis in a two dimensional space is  a rotation 
 $\bo(\phi)$. Thus the matrices of multiplication by $\fz$ with respect to 
 $\{\ff^{\alpha}_{n}(\z;\bc_1,\dots,\bc_{g})\}$ and $\{\ff^{\alpha+\mu_{\bc_g}}_{n}(\z;\bc_g,\dots,\bc_{g-1})\}$ are related 
  by \eqref{jfp}. It remains to compute the angle $\phi$ by means of $A(\alpha;\bC)$. Since 
  $A(\alpha+\mu_{\bc_g};\bc_g,\dots,\bc_{g-1})\in\bbA$, we have
  $$
  \bo(\phi)^*\begin{bmatrix}
0&0\\
p_{g-1}(\alpha)&p_g(\alpha)
\end{bmatrix}\bo(\phi)
=
\begin{bmatrix}
*&0\\
*&0
\end{bmatrix}.
  $$
  That is, $p_{g-1}(\alpha)\cos\phi-p_g(\alpha)\sin\phi=0$. Since $p_g(\alpha)>0$ we obtain \eqref{jfp0} for $\bo(\phi)$ of the form 
  \eqref{oofphi}.
\end{proof}
\begin{theorem}\label{th211}
In the above notations
\begin{equation}\label{perjac}
\cJ A(\alpha;\bC)=\cO^{\circ g} A(\alpha;\bC)
\end{equation}
\end{theorem}
\begin{proof}
We use \eqref{jfp}, having in mind that $\sum_{j=1}^g\mu_{\bc_j}=-\mu$  and that after all permutations we obtain the original ordering $\bC$.
\end{proof}

The  next lemma allows us to estimate components of the vector $\ff^{\alpha}_j$, $j=0,..,g,$ in its decomposition with respect to the  basis $\{\fe^\alpha_n\}_{n\ge 0}$.

\begin{lemma}
Let
$\ff_j^\alpha=\sum_{k=0}^\infty F^j_k(\alpha) \fe_k^\alpha$.
Then
\begin{equation}\label{estF}
|F^j_k(\alpha)|\le C(E) \eta^k, \quad j=0,\dots,g,
\end{equation}
where $1>\eta>\max\{|\fb(\z_1)|,\dots,|\fb(\z_g)|\}$.
\end{lemma}

\begin{proof}
First of all, we note that
$
\underline C(E) \le \|k^\alpha_{\zeta_n}\|\le \overline C(E)
$
uniformly on $\alpha\in\Gamma^*$.
Also $|\fb_{\bc_n}(\z_j)|\ge \underline c(E)$ and
by definition \eqref{defeal}, $|\fe^\alpha_k(\z_n)|\le \overline c(E)\eta^k$.
Since
\begin{eqnarray*}
\langle \fe^\alpha_k,
 \prod_{n=1}^{j-1}\fb_{\bc_n}
{k_{\zeta_j}^{\beta_j}}\rangle
&=&\langle 
 \frac{\fe^\alpha_k}{ \prod_{n=1}^{j-1}\fb_{\bc_n}}-\sum_{n=1}^{j-1}
\frac{k^{\beta_j+\mu_{\bc_n}}_{\zeta_n}\fe^\alpha_k(\z_n)}{\fb_{\bc_n}k^{\beta_j+\mu_{\bc_n}}_{\zeta_n}(\z_n)
 \prod_{l=1,l\not=n}^{j-1}\fb_{\bc_l}(\z_n)}
,{k_{\zeta_j}^{\beta_j}}\rangle
\\
&=&\frac{\fe^\alpha_k(\z_j)}{ \prod_{n=1}^{j-1}\fb_{\bc_n}(\z_j)}-\sum_{n=1}^{j-1}
\frac{k^{\beta_j+\mu_{\bc_n}}_{\zeta_n}(\z_j)\fe^\alpha_k(\z_n)}{\fb_{\bc_n}(\z_j)k^{\beta_j+\mu_{\bc_n}}_{\zeta_n}(\z_n)
 \prod_{l=1,l\not=n}^{j-1}\fb_{\bc_l}(\z_n)}
\end{eqnarray*}
and this is $ e^{-\pi i\beta_j(\oc\gamma_j)}{\|k_{\zeta_j}^{\beta_j}\|}\overline{F_k^j(\alpha)}$,  we get \eqref{estF}.
\end{proof}

\subsection{Transfer matrix}

In this subsection we discuss briefly the direct spectral problem of the class $A(E,\bC)$. For a vector $\vec{v}=\vp,\vq$, we  use the following notations
\begin{equation}\label{updown}
(u_k\vec{v})^*=\begin{bmatrix} v_0&\hdots & v_{g-k} \end{bmatrix},\quad 
(d_k\vec{v})^*=\begin{bmatrix} v_k& \hdots & v_{g} \end{bmatrix}.
\end{equation}

 Recall, $\{\delta_j\}_{j=0}^g$ is the standard basis in $\bbC^{g+1}$. We define inductively    upper triangular matrices $M_j$'s of dimension $(g+1-j)\times(g+1-j)$ such that
\begin{equation}\label{mdef0}
B(\vbp)-\vp(\vq)^*=M(\vbp):=M_0=\begin{bmatrix}M_1& 0\\
0 & 0\end{bmatrix}+(-\vp\, q_g+\vq\, p_g)\delta_g^*
\end{equation}
and
\begin{equation}\label{mdefj}
M_j=\begin{bmatrix}M_{j+1}& 0\\
0 & \bc_{g+1-j}\end{bmatrix}+(-u_j\vp\, q_{g-j}+{u_j\vq}\, p_{g-j})\delta_{g-j}^*, \ j\ge 1.
\end{equation}

 \begin{theorem}\label{th213}
 Let
 \begin{equation}\label{forfa1}
\begin{bmatrix}
 R_{00}(z)& R_{0g}(z)\\
  R_{g0}(z)& R_{gg}(z)
\end{bmatrix}=
\begin{bmatrix}
\langle(B_0-z)^{-1}\vp,\vp\rangle& \langle(B_0-z)^{-1}\delta_g,\vp\rangle \\
 \langle(B_0-z)^{-1}\vp, \delta_g\rangle& \langle(B_0-z)^{-1}\delta_g,\delta_g\rangle
\end{bmatrix}
\end{equation}
and
$
r_+(z)=\|\vp\|^2\langle(A_+-z)^{-1}\tilde e_0, \tilde e_0\rangle.
$
Then the shift by one block for  a one-sided GMP matrix $A_+\mapsto A_+^{(1)}$, see \eqref{gsmpshift}, by means of the spectral function has the following form
\begin{equation*}\label{forfa3}
r_+(z)=\frac{\fA_{11}(z) r^{(1)}_+(z)+\fA_{12}(z)}{\fA_{21}(z)r^{(1)}_+(z)+\fA_{22}(z)},
\end{equation*}
where
\begin{equation}\label{forfa2}
\fA(z):=\begin{bmatrix}
\fA_{11}&\fA_{12}\\
\fA_{21}&\fA_{22}
\end{bmatrix}(z)=
\frac{1}{R_{0g}(z)}\begin{bmatrix}
R_{00}R_{gg}-R_{0g}^2& -R_{00}\\
R_{gg}&-1
 \end{bmatrix}(z).
\end{equation}
 \end{theorem}
\begin{proof}
We represent $A_+$ as a two dimensional perturbation of the block diagonal matrix
\begin{equation}\label{gsmpshift}
A_+=\begin{bmatrix}
B(\vbp)&0\\ 0 &A_+^{(1)}
\end{bmatrix}+\|\vp\,^{(1)}\|(e_g\langle \cdot, \tilde e^{(1)}_0 \rangle+\tilde e^{(1)}_0\langle \cdot, e_g \rangle)
\end{equation}
and apply the resolvent perturbation formula.
\end{proof}
Note that in the definition \eqref{forfa2} we use the normalization $\det\fA(z)=1$.
\begin{definition}
Let $\bp^*=\begin{bmatrix}p&q\end{bmatrix}\in \bbR^2$. The matrix function
\begin{equation}\label{facto2}
\fa(z,\bc;\bp)=I-\frac{1}{\bc-z}\bp \bp^*\fj=e^{-\frac{1}{\bc-z}\bp\bp^*\fj},\quad
\fj=
\begin{bmatrix}
0&-1\\
1&0
\end{bmatrix}
\end{equation}
represents the so-called Blaschke-Potapov factor of the third kind with a real pole $\bc$ \cite{POT}. 
Note that $(\bp\bp^*\fj)^2=0$.
A specific factor related to infinity we introduce in the form
\begin{equation}\label{thematrix}
\fa(z;\bp)=\fa(z,\infty;\bp)=
\begin{bmatrix}
0&-{p}\\
\frac 1{p}&\frac{z-pq}{p}
\end{bmatrix}.
\end{equation}
\end{definition}

\begin{theorem}
Let $\bp_j^*=\begin{bmatrix}p_j&q_j\end{bmatrix}$. The matrix function $\fA(z)$, given in \eqref{forfa2}, possesses the following multiplicative representation
\begin{equation}\label{facto3}
\fA(z)=\fa(z,\bc_1;\bp_0)\fa(z,\bc_2;\bp_1)\dots\fa(z,\bc_g;\bp_{g-1})\fa(z;\bp_g).
\end{equation}
\end{theorem}

\begin{proof} We use the representation \eqref{mdef0} and definitions \eqref{forfa1}, \eqref{thematrix} to get
$
\fA(z)=\fA_0(z)\fa(z;\bp_g),
$
where
\begin{equation}\label{facto1}
\fA_0(z)=I-\begin{bmatrix} \langle (M_1-z)^{-1} {u_1\vp},{u_1\vp}\rangle&
\langle (M_1-z)^{-1}{u_1\vq},{u_1\vp}\rangle\\
\langle (M_1-z)^{-1}{u_1\vp},{u_1\vq}\rangle & \langle (M_1-z)^{-1}{u_1\vq},{u_1\vq}\rangle
\end{bmatrix}\fj.
\end{equation}
Recall, that $u_j\vp, u_j\vq$ were defined in \eqref{updown}.
Then, we use one after another \eqref{mdefj} and definitions \eqref{facto2} to get 
$$
\fA_{j-1}(z)=\fA_{j}(z)\fa(z,\bc_{g+1-j};\bp_{g-j}),
$$
where
$$
\fA_{j-1}(z)=I-\begin{bmatrix} \langle (M_j-z)^{-1}{u_j\vp},{u_j\vq}\rangle&
\langle (M_j-z)^{-1}{u_j\vq},{u_j\vp}\rangle\\
\langle (M_j-z)^{-1}{u_j\vp},{u_j\vq}\rangle & \langle (M_j-z)^{-1}{u_j\vq},{u_j\vq}\rangle
\end{bmatrix}\fj.
$$
That is, we obtain \eqref{facto3}. 
\end{proof}

\begin{definition} Let 
$\oc A\in A(E,\bC)$.
Then the product \eqref{facto3}
is called the \textit{transfer matrix} associated  with the given $\oc A$.
\end{definition}

The role of the transfer matrix is described in the following theorem.

\begin{theorem}\label{inth7}
Let $\oc A\in A(E,\bC)$  with the transfer matrix $\fA(z)$, given in \eqref{facto3}, and
let $\Delta(z):=\tr\, \fA(z)$. Then the spectrum $E$ of $\oc A$ is given by 
\begin{equation}\label{spectrum}
E=\Delta^{-1}([-2,2])=\{x:\, \Delta(x)\in[-2,2]\}.
\end{equation}
Moreover, $\Delta(z)$ is of the form \eqref{eq15},
where
\begin{equation}\label{explpqg}
\lambda_0 p_g=1, \quad {\lambda_0}\sum_{j=0}^{g}p_j q_j+{\bc_0}=0,
\end{equation}
and $\lambda_k=\Lambda_k(\vbp)
:=-\Res_{\bc_k}\tr \fA(z)$, i.e.,
\begin{equation}\label{explrhk}
\lambda_k=
-\tr\{\prod_{j=0}^{k-2}\fa(\bc_k,\bc_{j+1};\bp_j)\bp_{k-1}
\bp^*_{k-1}\fj
\prod_{j=k}^{g-1}\fa(\bc_k,\bc_{j+1};\bp_j)\fa(\bc_k;\bp_g)\}.
\end{equation}
\end{theorem}

\begin{proof}
A proof of \eqref{spectrum} is the same as  in the case of periodic Jacobi matrices. The relations
\eqref{explpqg} and \eqref{explrhk} follow immediately from \eqref{facto3}.
\end{proof}

\begin{proof}[Proof of Proposition \ref{prop18}]
First of all, we have a parametrization of $A(E,\bC)$ by the characters $\Gamma^*$.
It is evident that, in the basis \eqref{gsmpbase}, multiplication by $\Psi$  is the shift $S^{g+1}$, $\Psi \ff_n^{\alpha}=\ff^\alpha_{n+(g+1)}$. Thus, the magic formula for GMP matrices corresponds to the definition \eqref{eq15}. The relations
\eqref{explpqg}, \eqref{explrhk} imply the form of the isospectral surface, that is, \eqref{iso101}.
Conversely, if $\Delta(\oc A)=S^{-(g+1)}+S^{g+1}$, then $\oc A$ is periodic by
Na\u{\i}man's Lemma \cite[Lemma 3.4]{KSDp}  (alternatively, see proof of Theorem \ref{th8.4}). Therefore we can apply Theorem \ref{inth7}.
\end{proof}

Later, in Section 6, we will use another representation for $q_g$.

\begin{lemma}\label{lem:c0formula}
$q_g$ allows the following alternative representation 
\begin{equation}
q_g+\bc_0
=\sum_{k=1}^{g}\tr \{\prod_{j=0}^{k-2}\fa(\bc_k,\bc_{j+1};\bp_j)\bp_{k-1}
\bp_{k-1}^*\fj
\prod_{j=k}^{g-1}\fa(\bc_k,\bc_{j+1};\bp_j)\begin{bmatrix}
0& 0\\ 0&\frac 1 {p_{g}}
\end{bmatrix}\}.\label{alternativqg}
\end{equation}
\end{lemma}
\begin{proof}
From the second relation in \eqref{explpqg} and \eqref{facto1} one has
$$
q_g+\bc_0=\frac 1{2\pi i}\oint_{|z|=R}\tr\, \fA_0(z)\begin{bmatrix}
0& 0\\ 0&\frac 1{p_g}
\end{bmatrix}dz=\sum_{k=1}^g\Res_{\bc_k}
\tr\, \fA_0(z)\begin{bmatrix}
0& 0\\ 0&\frac 1{p_g}
\end{bmatrix},
$$
which is \eqref{alternativqg}.
\end{proof}

\section{GMP matrices, general case.}\label{Section3}
 
 We hope after Theorem \ref{inth7}, and especially \eqref{explrhk}, it would be easy to perceive the following notations. 
\begin{notation}\label{notla} For $k=1,\dots, g$ the following functions (polynomials) are given by 
\begin{align}\label{deflad}
\Lambda^\#_{j,k}=&\Lambda^\#_k(\bp^{(j+1)}_0,\dots,
\bp^{(j+1)}_{k-1};\bp^{(j)}_{k-1},\dots, \bp^{(j)}_g) \\
=&-\tr\{
\prod_{m=0}^{k-2}\fa(\bc_k,\bc_{m+1};\bp^{(j+1)}_m)
\bp^{(j+1)}_{k-1}
(\bp^{(j)}_{k-1})^*\fj
\prod_{m=k}^{g-1}\fa(\bc_k,\bc_{m+1};\bp^{(j)}_m)
\fa(\bc_k;\bp^{(j)}_g)\}\nonumber
\end{align}
If, as before, $\bp^{(j)}_m=\bp^{(j+1)}_m=\bp_m$  for all $m\in [0,g]$  this notation is simplified to
\begin{equation}\label{defla}
\Lambda_k(\vbp)=\Lambda^\#_k(\bp_0,\dots,
\bp_{k-1};\bp_{k-1},\dots, \bp_g)=-\Res_{\bc_k}\tr \fA(z).
\end{equation}
\end{notation}
\begin{lemma}\label{lem:gsmpEntries}
Let $A\in \bbA$. 
Then the formal inverse to $\bc_k-A$ is well defined as soon as 
$\Lambda^\#_{j,k}\not=0$. Moreover, each element of the inverse matrix is a rational function of the coefficients $\vbp_j,\vbp_{j+1}$ of the two consecutive blocks of $A$ with the denominator $\Lambda^\#_{j,k}$.
In particular, the  vector $f(\bc_k)$ such that $(\bc_k-A)f(\bc_k)=e_{k-1}$, $k=1,..,g$, with the vector-components
\begin{equation*}\label{egsmp21}
f_j=f_j(\bc_k)=\{(f^{(j)}(\bc_k))_n\}_{n=0}^g\in\bbC^{g+1},
\end{equation*}
 obey $f_j(\bc_k)=0$ for $j\not\in\{-1,0,1\},$  
 $$
 (f^{(-1)})_0=...=(f^{(-1)})_{k-2}=0,\quad (f^{(1)})_k=...=(f^{(1)})_{g}=0,
 $$ 
 and
\begin{align}
{\Lambda^\#_{-1,k}}(f^{(-1)})_{k-1}&=1,\quad {\Lambda^\#_{0,k}}(f^{(1)})_{k-1}=1, \label{eqsmp21a}\\
{\Lambda^\#_{-1,k}}(f^{(-1)})_l&=\frac{(\bp^{(-1)}_{k-1})^*\fj
\prod_{j=k}^{l-1}\fa(\bc_k,\bc_{j+1};\bp_j^{(-1)})
{\bp_{l}^{(-1)}}}{\bc_k-\bc_{l+1}},\ l=k,\dots,g-1, \label{eqsmp21b}\\
{\Lambda^\#_{0,k}}(f^{(1)})_{m}&=\frac{(\bp^{(1)}_{m})^*\fj
\prod_{j=m+1}^{k-2}\fa(\bc_k,\bc_{j+1};\bp_j^{(1)})
{\bp_{k-1}^{(1)}}}{\bc_k-\bc_{m+1}},\ m=0,\dots,k-2.\label{eqsmp21d}
\end{align} 
\end{lemma}
\begin{proof}
We have  a purely linear algebra problem. 
To find $f(\bc_k)$ we solve the  system
\begin{eqnarray}
-A(\vp_{-1})f_{-1}&=&0, \nonumber \label{egsmp22}\\
(\bc_k-B(\vbp_{-1}))f_{-1}-A_0f_0&=&0, \nonumber\label{egsmp23}\\
-A^*(\vp_0)f_{-1}+(\bc_k-B(\vbp_{0}))f_{0}-A(\vp_1)f_1&=&\delta_{k-1},\nonumber\label{egsmp24}\\
-A^*(\vp_1)f_{0}+(\bc_k-B(\vbp_{1}))f_{1}&=&0,\nonumber\label{egsmp25}\\
-A^*(\vp_2)f_{1}&=&0.\label{egsmp26}
\end{eqnarray}
It is worth to recall that $B(\vbp)$ is an upper triangular matrix $M(\vbp)$ up to a one-dimensional perturbation, see \eqref{mdef0}, and its main diagonal in this case is $\tilde \bC$, see definition \eqref{n2}. For this reason all inverse matrices to $(\bc_k-B(\vbp))$ can be found exactly like in the previous section in terms of products of Blaschke-Potapov factors $\fa(z,\bc;\bp)$.
\end{proof}

\begin{theorem}\label{defaltdef}
Let $A\in\bbA$. $A$ belongs to the GMP class if and only if the forming sequences $\{\vp_j,\vq_j\}$ satisfy the following conditions
\begin{equation}\label{altdef}
\inf_{j\in\bbZ}
\Lambda^\#_k(\bp^{(j+1)}_0,\dots,
\bp^{(j+1)}_{k-1};\bp^{(j)}_{k-1},\dots, \bp^{(j)}_g)
>0, \ \text{for all}\ k=1,\dots,g.
\end{equation}
\end{theorem}
\begin{proof}\label{thstr} 
The formal inverse to $\bc_k-A$ can be found explicitly, see the previous lemma. Moreover, 
solvability of the system \eqref{egsmp26} is equivalent to \eqref{eqsmp21a}. 
Thus, from one side, if $\bc_k -A$ is invertible  we have 
$$
\frac 1{\Lambda_{j,k}^\#}=|\langle (\bc_k-A)^{-1}e_{(j+1)(g+1)+k-1}, e_{j(g+1)+k-1} \rangle|\le \|(\bc_k-A)^{-1}\|.
$$
That is, \eqref{altdef} holds.

In the opposite direction,
the formal inverse operators exists since $\Lambda_{j,k}^\#\not=0$ for all $k$ and $j$.
The following estimation is very useful
\begin{equation}\label{oppalt1}
\|\fa(\bc_k,\bc_{j};\bp)\|\le e^{\frac{\|\bp\|^2}{\inf_{k\not=j}|\bc_k-\bc_j|}},
\
\text{
due to}\ 
\fa(\bc_k,\bc_{j};\bp)=e^{-\frac{1}{\bc_k-\bc_j}\bp\bp^*\fj}.
\end{equation}
It shows that every non-trivial entry  $\langle (\bc_k-A)^{-1}e_j,e_n \rangle$  is bounded  by
$
C_1e^{\frac{C_2(g)\|A\|^2}{\inf_{k\not=j}|c_k-c_j|}}
$, see \eqref{eqsmp21a}-\eqref{eqsmp21d}. In this estimation  $C_1>0$ depends only on the infimum \eqref{altdef}. 
Since the formal inverse has only $(2g+3)$ non-trivial diagonals, we proved that  all $\bc_k-A$ are indeed invertible, i.e., $\|(\bc_k-A)^{-1}\|<\infty$.
\end{proof}

\section{Jacobi flow, general case}
Let us mention once again that Theorem \ref{th211} gives already a certain hint for a constructive definition of the Jacobi flow. It will be  defined via the unitary transformation, which after $g$ rotations and one shift, maps GMP$(\bc_1,\dots,\bc_g)$ 
into itself.
The first rotation creates the matrix $\tilde A$, which belongs (up to a suitable shift) to GMP$(\bc_g,\bc_1,...,\bc_{g-1})$ class. Then we create a matrix of the class GMP$(\bc_{g-1},\bc_g,\bc_1,...,\bc_{g-2})$, and so on... On the last step (making the shift) we get the required  Jacobi flow transform, see \eqref{jflow}.
Having in  mind \eqref{eqn:ufunctional} and \eqref{jfp0}, we give the following definition.

\begin{definition}\label{defo}
We  define the  map 
$$
\cO:\text{GMP}(\bc_1,\bc_2,...,\bc_{g})\to\text{GMP}(\bc_g,\bc_1,...,\bc_{g-1})
$$ in the following way. Let 
$O=O_A$ be the block-diagonal matrix
$$
O=\begin{bmatrix} \ddots& & & \\ & O_{-1}& &\\
& & O_{0} & \\
& & &\ddots
\end{bmatrix}
$$
where $O_k$ are the $(g+1)\times(g+1)$ orthogonal matrices, see \eqref{oofphi},
$$
O_k=\begin{bmatrix}I_{g-2}& 0\\
0&\bo(\phi_k)
\end{bmatrix}, \ 
\begin{bmatrix}
\sin\phi_k&
\cos\phi_k
\end{bmatrix}
=\frac{
\begin{bmatrix}
p^{(k)}_{g-1}&
p^{(k)}_{g}
\end{bmatrix}}{\sqrt{(p^{(k)}_{g-1})^2+(p^{(k)}_{g})^2}}.
$$
Then
\begin{equation}\label{varpidef}
\cO A:=S O_A^* AO_A S^{-1}.
\end{equation}
\end{definition}
It is required, but easy to check  the correctness of this definition. Note that
 for $p$-entries of $\tilde A=\cO A$ we get
\begin{equation}\label{def3}
\tilde p^{(0)}_j=p^{(0)}_{j-1}\cos\phi_{-1}, \quad 1\le j\le g-1;
\end{equation}
\begin{equation}\label{def2}
\tilde p^{(0)}_g=\sqrt{(p^{(0)}_{g-1})^2+(p^{(0)}_{g})^2}\cos\phi_{-1}=
\sqrt{\frac{(p^{(0)}_{g-1})^2+(p^{(0)}_{g})^2}{(p^{(-1)}_{g-1})^2+(p^{(-1)}_{g})^2}}p^{(-1)}_g.
\end{equation}
Also,
\begin{equation*}\label{def0}
\begin{bmatrix}\tilde q_g^{(-1)}\tilde p_g^{(-1)}& \tilde p^{(0)}_0\\
\tilde p_0^{(0)}& \tilde q_0^{(0)}\tilde p_0^{(0)}+\bc_g\end{bmatrix}
=
\bo(\phi_{-1})^*
\begin{bmatrix}q^{(-1)}_{g-1}p^{(-1)}_{g-1}+\bc_g&q^{(-1)}_{g-1}p^{(-1)}_{g}\\
q^{(-1)}_{g-1}p^{(-1)}_{g}&q^{(-1)}_{g} p^{(-1)}_{g}
\end{bmatrix}
\bo(\phi_{-1}).
\end{equation*}
Thus, the $q$-entries have the form
\begin{equation}\label{def1}
\tilde q^{(0)}_0\tilde p^{(0)}_g=-\sin\phi_{-1}\sqrt{(p^{(0)}_{g-1})^2+(p^{(0)}_{g})^2},\quad 
\tilde q_{j}^{(0)}\tilde p^0_g=q_{j-1}^{(0)} \sqrt{(p^{(0)}_{g-1})^2+(p_g^{(0)})^2}.
\end{equation}

Our next definition is a counterpart of \eqref{perjac}.

\begin{definition}\label{defjf}
We  define the  Jacobi flow transform 
$$
\cJ:\text{GMP}(\bc_1,\bc_2,...,\bc_{g})\to\text{GMP}(\bc_1,\bc_2,...,\bc_{g})
$$ by 
\begin{equation}\label{jflow}
\cJ A=S^{-(g+1)}\cO^{\circ g} AS^{g+1}=\cO^{\circ g}(S^{-(g+1)} AS^{g+1}).
\end{equation}
\end{definition}

Let us  note that 
\begin{equation}\label{comvarpi}
S^{-(g+1)}\cO(A) S^{g+1}=\cO(S^{-(g+1)}A S^{g+1}).
\end{equation}
This has an important consequence.

\begin{corollary}\label{corshift}
\begin{equation}\label{shiftjflow}
\cO(\cJ^{\circ n} A)=\cJ^{\circ n}(\cO A).
\end{equation}
\end{corollary}
\begin{proof} Due to \eqref{jflow} and \eqref{comvarpi} we get
$$
\cJ (\cO A)=\cO^{\circ g}(S^{-(g+1)} \cO AS^{g+1})=\cO^{\circ (g+1)}(S^{-(g+1)} AS^{g+1})=\cO(\cJ A).
$$
\end{proof}

Let us turn to explicit formulas for the given transform. First of all, we note that
\begin{equation}\label{jfex0}
\cJ A=S^{-1}U^*_A A U_A S,
\end{equation}
where  $U_A$ is a $(g+1)\times (g+1)$-block diagonal matrix
$$
U_A=\begin{bmatrix} \ddots& & & \\ & U(\vp_{-1})& &\\
& & U(\vp_0) & \\
& & &\ddots
\end{bmatrix}.
$$
\begin{lemma}\label{lem:formulaU}
In the notation above, see also \eqref{updown}, we have 
\begin{align}\label{eq:formulaU}
U(\vp)\delta_0=\frac{1}{\|\vp\|}\vp,\quad
U(\vp)\delta_k=\frac{1}{\|d_{k-1}\vp\|\|d_{k}\vp\|}
\begin{bmatrix}
0\\\|d_{k}\vp\|^2\\-p_{k-1}d_k\vp
\end{bmatrix}, 1\leq k\leq g.
\end{align}
\end{lemma}
\begin{proof}
It follows from a step by step representation of the block $U(\vp)$ as   the product of orthogonal matrices,
see \eqref{jflow},
\begin{equation*}
U(\vp)=\begin{bmatrix}I_{g-2}& 0\\
0& \bo(\phi_g)
\end{bmatrix}
\dots
\begin{bmatrix}
\bo(\phi_1)& 0\\
0& I_{g-2}
\end{bmatrix},\ \begin{bmatrix}
\sin\phi_k&\cos\phi_k
\end{bmatrix}=\frac{\begin{bmatrix}p_{k-1}&\|d_k\vp\|
\end{bmatrix}}{\|d_{k-1}\vp\|}.
\end{equation*}
\end{proof}

\begin{theorem}\label{th54}
Let $A(1)=\cJ A$ and let $\{p^{(j)}_k(1),  q_k^{(j)}(1)\}$ be generating coefficient sequences of $A(1)$. Then
\begin{align}\label{jfex}
\begin{bmatrix}
 q_0^{(j)}(1)\\
\vdots\\
q_{g-1}^{(j)}(1)
\end{bmatrix}
= 
\|\vp_{j}\|
\begin{bmatrix}
\vdots\\
-\frac{p_k^{(j)}}{\|d_k\vp_j\|\|d_{k+1}\vp_j\|}\\
\vdots
\end{bmatrix},\ 
\begin{bmatrix}
*\\  p_0^{(j)}(1)\\ \vdots \\ p_{g-1}^{(j)}(1)
\end{bmatrix}
=
U^*(\vp_j)B(\vbp_{j})\frac{\vp_j}{\|\vp_j\|},\\ \label{jfex1}
p_g^{(j)}(1)
=
\frac{\|\vp_{j+1}\|}{\|\vp_{j}\|}p_g^{(j)}, \ 
q_g^{(j)}(1)
=
\frac {\|\vp_{j}\|}{p_g^{(j)}\|\vp_{j+1}\|}
\frac{\langle B(\vbp_{j+1})\,\vp_{j+1},  \vp_{j+1}\rangle}{\|\vp_{j+1}\|^2 }.
\end{align}
\end{theorem}
\begin{proof}
We get \eqref{jfex} and \eqref{jfex1} from \eqref{jfex0} by Lemma \ref{lem:formulaU}.
\end{proof}

\begin{remark}
In view of Theorem \ref{th54} the Jacobi flow on GMP matrices can be related to an open (input-output) dynamical system, see \eqref{diag:jacobiflow}. Let us fix a block-position $j=0$, but vary $n$ in $A(n+1)=\cJ A(n)$. Then the coefficients related to the next block $j=1$ are involved only in \eqref{jfex1} and in a very specific way. If we define the \textit{two dimensional input} by
$$
a^{\text{in}}(n)=\|\vp_1(n)\|,\quad b^{\text{in}}(n)=\frac{\langle B(\vbp_{1}(n))\,\vp_{1}(n),  \vp_{1}(n)\rangle}{\|\vp_{1}(n)\|^2 }
$$
and consider the parameters $\{\vp_{0}(n),\vq_{0}(n)\}$ as the \textit{internal state} of the system, then the open dynamical system is defined by
\begin{equation}\label{ods14}
\begin{bmatrix}
b^{\text{out}}(n)&\vp_{0}(n+1)^*\\
\vp_{0}(n+1)&B(\vbp_{0}(n+1))
\end{bmatrix}=
\begin{bmatrix}
U^*(\vp_{0}(n))&0\\
0&1
\end{bmatrix}
\begin{bmatrix}
B(\vbp_{0}(n))&\delta_g a^{\text{in}}(n)\\
a^{\text{in}}(n)\delta_g^*&b^{\text{in}}(n)
\end{bmatrix}
\begin{bmatrix}
U(\vp_{0}(n))&0\\
0&1
\end{bmatrix},
\end{equation}
 and $a^{\text{out}}(n+1)=\|\vp_{0}(n+1)\|$. 
 Note also that the \textit{output} $\{a^{\text{out}}(n+1),b^{\text{out}}(n)\}$ are the Jacobi parameters of 
 $J=\cF A$ (cf. \eqref{coefflow}, \eqref{nis50}) and the input is related to $\cF S^{-(g+1)}AS^{g+1}$. That is, this system represents the \textit{GMP transform on Jacobi matrices}.
\end{remark}

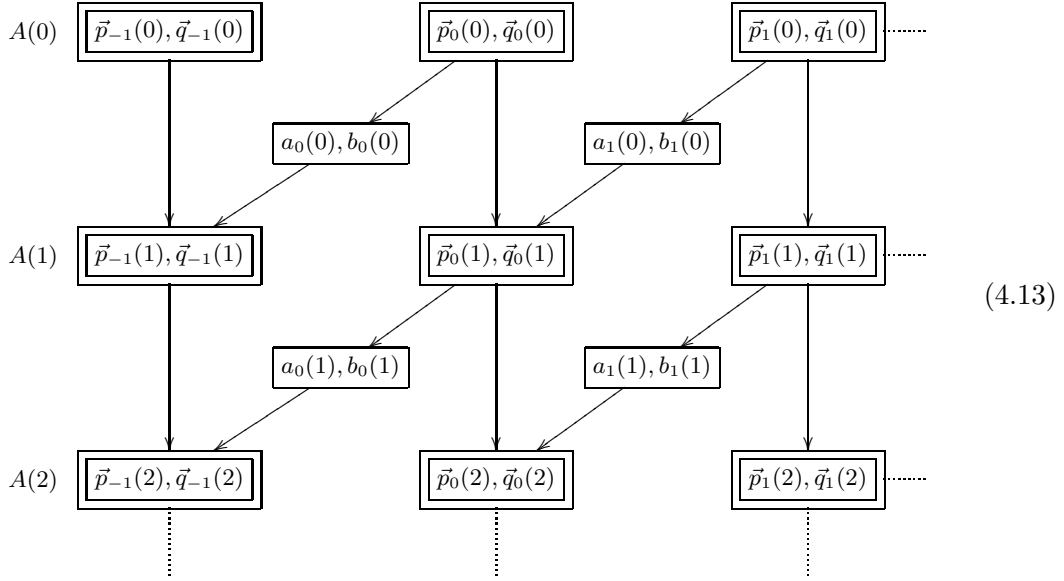
\begin{equation}\label{diag:jacobiflow}
\begin{aligned}
	{\footnotesize
	\xymatrix@C=0.5em{ 
		A(0)&*+[F]{\boxed{\vp_{-1}(0),\vq_{-1}(0)}}\ar[dd]  & & *+[F]{\boxed{\vp_{0}(0),\vq_{0}(0)}}\ar[dl]\ar[dd] 
		&& *+[F]{\boxed{\vp_{1}(0),\vq_{1}(0)}}\ar[dl]\ar[dd]\ar@{.}[rr]&&~ \\
		&& *+[F]{a_0(0),b_{0}(0)}\ar[dl]&&*+[F]{a_1(0),b_{1}(0)}\ar[dl]&\\
		A(1)&*+[F]{\boxed{\vp_{-1}(1),\vq_{-1}(1)}}\ar[dd]  & & *+[F]{\boxed{\vp_{0}(1),\vq_{0}(1)}}\ar[dl]\ar[dd] 
		&& *+[F]{\boxed{\vp_{1}(1),\vq_{1}(1)}}\ar[dl]\ar[dd]\ar@{.}[rr]&&~ \\
		&&
		*+[F]{a_0(1),b_{0}(1)}\ar[dl]&&*+[F]{a_1(1),b_{1}(1)}\ar[dl]&\\
		A(2)&*+[F]{\boxed{\vp_{-1}(2),\vq_{-1}(2)}} \ar@{.}[d] & & *+[F]{\boxed{\vp_{0}(2),\vq_{0}(2)}}\ar@{.}[d]
		&& *+[F]{\boxed{\vp_{1}(2),\vq_{1}(2)}} \ar@{.}[rr]\ar@{.}[d]&&~\\
		&&&&&&&
	}
}
\end{aligned}
\end{equation}

\begin{theorem}
Let $A(0):=A\in \GMP(\bC)$, $A(n+1)=\cJ A(n)$, $n\in\bbZ$, and $\tilde e_{-1}:=e_{-1}$. Define
\begin{eqnarray}
\tilde e_m&=&U_{A(0)}S U_{A(1)}S\cdots U_{A(m)} S\tilde e_{-1}, \quad m\ge 0,\label{nis51}\\
\tilde e_{m-1}&=&S^{-1}U_{A(-1)}^{-1}\cdots S^{-1} U_{A(m)}^{-1}\tilde e_{-1}, \quad m<0.\label{nis52}
\end{eqnarray}
These system of vectors form an orthonormal system in $\l^2$, with respect to which the following three-term recurrence relation holds
\begin{equation}\label{nis50}
A\tilde e_{m-1}=a(m-1)\tilde e_{m-2}+b(m-1)\tilde e_{m-1}+a(m)\tilde e_{m},
\end{equation}
where $a(m)$ and $b(m-1)$ 
are given  by
\begin{equation}\label{zhz}
a(m)=\| \vp_0(m)\|=\sqrt{p^{(0)}_0(m)^2+
\dots+p^{(0)}_g(m)^2
}, \quad b(m-1)=p_g^{(-1)}(m)q_g^{(-1)}(m).
\end{equation}
That is, $A$ with respect to $\{\tilde e_m\}$ is a Jacobi matrix $J$, moreover the transformation \eqref{jfex0} corresponds to its shift
$S^{-1}JS$.
\end{theorem}

\begin{proof}
This relation for $m=0$ follows basically from the definition. We consider the term related to $\tilde e_{-2}$, which is the most nontrivial in this case. We have
$$
A=S^{-1}U_{A(-1)}^{-1}A(-1)U_{A(-1)}S.
$$
Therefore
$
SA e_{-1}=U_{A(-1)}^{-1}A(-1)U_{A(-1)}e_0.
$
Using the block structure of $U_A$ we obtain
\begin{equation}\label{zhy}
P_-S Ae_{-1}=U_{A(-1)}^{-1}P_-A(-1)U_{A(-1)}e_0=U_{A(-1)}^{-1}e_{-1}\|\vp_0(-1)\|.
\end{equation}
Having in mind \eqref{nis50} and $\tilde e_{-1}=e_{-1}$, we formally define 
$$
a(-1) \tilde e_{-2}:=S^{-1}P_-S Ae_{-1},\quad a(0)\tilde e_0:=P_+ A e_{-1}.
$$ Then, due to \eqref{zhy},
$$
a(-1) \tilde e_{-2}=\|\vp_0(-1)\| S^{-1}U_{A(-1)}^{-1}e_{-1},
$$
which proves both \eqref{nis52} and the first relation in \eqref{zhz} for $m=-1$.

We can write a similar relation for $A(1)$. Using
definition \eqref{jfex0}, we rewrite such a relation by means of the original $A$. As the result, we obtain  \eqref{nis50} for $m=1$, and so on. 
 Simultaneously, we proved Corollary \ref{cor113}.
\end{proof}

\begin{remark} In fact $\{\tilde e_n\}_{n=-\infty}^\infty$ is a basis in $\l^2$, see Proposition \ref{prop73}. That is,
$$
J=\cF A=F^* A F, \quad Fe_n=\tilde e_n,
$$
and $F$ is unitary. 
\end{remark}

\section{Spectral conditions}

\subsection{Killip-Simon spectral conditions for one- and two-sided Jacobi matrices}
First of all we mention the following 
\begin{lemma}\label{lemma51ini}
Assume that $J_+$ is a one-sided Jacobi matrix  with essential spectrum on $E$. Then it can be extended by a matrix $ J_-=P_-\oc J P_-$,
$\oc J\in J(E)$, such that each $\bc_j$ belongs to the resolvent set (domain) of the resulting matrix \eqref{ijf2}.
\end{lemma}
\begin{proof}
 Let
 \begin{equation}\label{sh1}
R(z):=\cE^*(J-z)^{-1}\cE=\int\frac{d\Sigma}{x-z}=\begin{bmatrix}
r_{-}(z)^{-1}& a(0)\\
a(0)& r_+(z)^{-1}
\end{bmatrix}^{-1},
\end{equation}
where $\cE:\bbC^2\to\l^2$ such that $\cE\begin{bmatrix}c_-\\ c_+\end{bmatrix}=c_-e_{-1}+c_+ e_0$. In particular, for the diagonal entries of $R(z)$ we have
\begin{equation}\label{inizp}
-\frac 1{R_{-1,-1}(z)}=-\frac 1{r_-(z)}+a(0)^2r_+(z),\
-\frac 1{R_{0,0}(z)}=-\frac 1{r_+(z)}+a(0)^2r_-(z).
\end{equation}
If $r_+(\bc_j)$ is zero or infinity, we choose $\oc J$ such that $r_{-}(\bc_j)$ is regular, that is, $r_{-}(\bc_j)\not=0$, $r_{-}(\bc_j)\not=\infty$. And vice versa, if 
$r_+(\bc_j)$ is regular, we set $r_-(\bc_j)=0$. In both cases $R_{-1,-1}(\bc_j)\not=\infty$ and $R_{0,0}(\bc_j)\not=\infty$. Therefore the whole matrix $(\bc_j-J)$ is invertible.
\end{proof}
The spectral Killip-Simon condition can be formulated either in terms of measures $\sigma_\pm$, see Definition \ref{defini1}, or by means of the matrix measure $d\Sigma$ from \eqref{sh1}.
\begin{lemma}
The measures $\sigma_\pm$ both satisfy the Killip-Simon condition if and only if the matrix measure $d\Sigma$ is supported on $E\cup Y$ and obeys
\begin{equation}\label{mmc}
\int_E|\log \det\Sigma'(x)|\sqrt{\dist(x,\bbR\setminus E)}dx+\sum_{y_k\in Y}\sqrt{\dist(y_k, E)}^3<\infty.
\end{equation}

\end{lemma}

\begin{proof}
We note two properties of an arbitrary function $F(z)$, which is analytic in the upper half-plan and has positive imaginary part. 
If such a function has a meromorphic extension in an interval $(\ba_j,\bb_j)\subset\bbR$ then its zeros and poles interlay. Secondary, $F(z)$ is of bounded characteristic in the upper half-plane, and therefore 
\begin{equation}\label{inioch}
\int_{\bbR}\frac{|\log F(x)|}{1+x^2}dx<\infty.
\end{equation}
From the first property we get that all poles of the first and second functions in \eqref{inizp} satisfy the Killip-Simon condition in $\bbR\setminus E$. Applying this fact once again we obtain that poles of $R_{-1,-1}$ and $R_{0,0}$, that is the set $Y$, satisfy this condition. Similar observations show the opposite directions.

With respect to the a.c. part of the measure we have
$$
\Sigma'(x)=\begin{bmatrix}r_-(x)^{-1}&a(0)\\
a(0)&r_+(x)^{-1}\end{bmatrix}^{-1}
\begin{bmatrix} {\sigma'}_-(x)&0\\
0&{\sigma'}_+(x)\end{bmatrix}
\begin{bmatrix}r_-(x)^{-1}&a(0)\\
a(0)&r_+(x)^{-1}\end{bmatrix}^{-1}.
$$
Therefore
$$
\det\Sigma'(x)=\left|\frac{r_-(x)}{-r_{+}^{-1}(x)+a(0)^2 r_-(x)}\right|^2\sigma'_-(x)\sigma'_+(x).
$$
Applying \eqref{inioch} to $r_-(z)$ and $-r_{+}^{-1}(z)+a(0)^2 r_-(z)$, we obtain an equivalence of the conditions for $\det\Sigma'(x)$ and $\sigma'_{\pm}(x)$.
\end{proof}

\subsection{Scalar and block-matrix spectral Killip-Simon conditions}

\begin{theorem}\label{thdensks} Let $A\in\GMP(\bC)$.
Its spectral measure satisfies \eqref{mmc} if and only if the block Jacobi matrix $\Delta(A)$ belongs to the Killip-Simon class.
\end{theorem}
 
Essentially, it follows from the  Lemma \ref{insl54} given below. We prove the corresponding lemma for a scalar measure $\sigma$, assuming that $\sigma(\bc_j)=0$. Note, even if we start with an initial one-sided matrix $J_+$ such that $\sigma_+(\bc_j)>0$ for some $j$, due to the Lemma \ref{lemma51ini}, we always can get a two-sided $J$ such that $\Sigma(c_j)=0$. 
  Also, it is more uniform to set
$
\Delta(z)=\sum_{j=1}^g\frac{\lambda_j}{\bc_j-z}
$. To pass to our case, where $\Delta(z)$ is of the form \eqref{eq15}, it is enough to send one of this $\bc_j$ to infinity by a suitable linear fractional transform.

So, let $d\sigma$ be a scalar measure with an essential support on $E=\Delta^{-1}([-2,2])$ such that $\sigma(\bc_j)=0$. We define the matrix measure $d\Xi$ by
\begin{equation}\label{nis3}
\int \frac{d\Xi(y)}{y-z}:=\int\frac 1 {\Delta(x)-z}W^*(x)d\sigma(x)W(x),
\end{equation}
where
\begin{equation}\label{nis4}
 W(x)=
\begin{bmatrix}
\frac{1}{\bc_1-x} &\hdots&\frac{1}{\bc_g-x}
\end{bmatrix}.
\end{equation}

In other words, $d\Xi$ is the matrix measure of the multiplication by $\Delta(x)$ in $L^2_{d\sigma}$
with respect to a suitable cyclic subspace.
 Note that one can normalize this measure by  a triangular (constant) matrix $L$ such that
$
L^*\int d\Xi(y) L=I,
$
that is, to choose an appropriate orthonormal basis in the fixed cyclic subspace.
\begin{lemma}\label{insl54}
Let $\Xi'(y)$ be the density of the a.c. part of the measure $d\Xi$ on $[-2,2]$ and
$\sigma'(x)$ be the density of $d\sigma$, respectively. Then
\begin{equation}\label{densks}
 \det\Xi'(y)=\frac{\prod_{\Delta(x)=y}\sigma'(x)}{\prod_{k=1}^g\lambda_k}.
\end{equation}
\end{lemma}

\begin{proof}
Let $\{x_1,\dots,x_g\}=\Delta^{-1}(y)$, $y\in[-2,2]$. Then
\begin{eqnarray*}
\Xi'(y)&=&\sum_{\Delta(x)=y}W^*(x)\frac{\sigma'(x)}{\Delta'(x)}W(x)\nonumber\\
&=&
W^*\begin{bmatrix}
\frac{\sigma'(x_1)}{\Delta'(x_1)}& & \\
&\ddots& \\
& & \frac{\sigma'(x_g)}{\Delta'(x_g)}
\end{bmatrix} W,\ W:=
\begin{bmatrix}
\frac{1}{\bc_1-x_1}&\dots&\frac 1{\bc_g-x_1}\\
\vdots&\dots&\vdots\\
\frac{1}{\bc_1-x_g}&\dots&\frac 1{\bc_g-x_g}
\end{bmatrix}.\label{nis54}
\end{eqnarray*}
As it is well known
\begin{equation}\label{densks1}
\det W
=(-1)^{\frac{g(g-1)}{2}}\frac{\prod_{k<j}(x_k-x_j)\prod_{k<j}(\bc_k-\bc_j)}{\prod_{j,k}(\bc_j-x_k)}.
\end{equation}
On the other hand,
$$
y-\Delta(x)=y\frac{\prod(x-x_j)}{\prod(x-\bc_j)}.
$$
Therefore,
$$
-\Delta'(x_k)=y\frac{\prod_{k\not=j}(x_k-x_j)}{\prod_j(x_k-\bc_j)} \quad\text{and}\quad
-\lambda_k=y\frac{\prod_j(c_k-x_j)}{\prod_{k\not=j}(\bc_k-\bc_j)}.
$$
That is,
$$
\Delta'(x_k)=\lambda_k\frac{\prod_{k\not=j}(x_k-x_j)}{\prod_j(x_k-\bc_j)}\frac{\prod_{k\not=j}(\bc_k-\bc_j)}{\prod_j(\bc_k-x_j)}.
$$
Thus,
\begin{equation}\label{densks2}
\prod \Delta'(x_k)=\frac{\prod_{k<j}(x_k-x_j)^2(\bc_k-\bc_j)^2}{\prod_{k,j}(\bc_k-x_j)^2}\prod_k\lambda_k.
\end{equation}
Combining \eqref{densks1} and \eqref{densks2}, we obtain \eqref{densks}.
\end{proof}

Next, we prove the following general statement.

\begin{proposition}\label{prop73}
Let a two-sided Jacobi matrix $J$ be such that $\bc_j\not\in\sigma(J)$. Then, up to the identification
$(p^{(j)}_m,q^{(j)}_m)\simeq(-p^{(j)}_m,-q^{(j)}_m)$, there exists a unique GMP matrix $A$  related to the fixed ordering $\bC$ such that $J=\cF A=F^*AF$. In particular $F:\l^2\to \l^2$ is unitary.
\end{proposition}

We need to define a counterpart of a cyclic subspace \eqref{smpbase0} in the general case.
Assume that $J=\cF A$, see \eqref{ijf2}-\eqref{ijf3}. Recall  $F:\l^2\to\l^2$ is the isometry 
$$
F e_m=\tilde e_m,
$$
where $\tilde e_m$ were defined in \eqref{nis51}-\eqref{nis52}.
In particular,
$Fe_{-1}=e_{-1}$ and  $FP_+=P_+ F$,  $Fe_0=\tilde e_0=\frac{1}{a(0)}P_+Ae_{-1}$.
We note that 
\begin{equation}\label{aort}
\{h=(A-\bc_1) f:\ f\in \l^2_+, \ \langle f,\tilde e_0 \rangle=0\}=\{h\in\l^2_+: \langle h, e_0 \rangle=0\}.
\end{equation}
Thus, $F^* e_0$ can be described by means of an orthogonal complement in the following construction.

Let $\bc\not\in \sigma(J)$ and, actually, it is not necessary, but let $\bc$ be real. 
Having in mind the previous paragraph,
we define
\begin{equation}\label{jort}
\l^2_{+,\bc}:=\{h=(J-\bc) f:\ f\in \l^2_+, \ \langle f,e_0 \rangle=0\}.
\end{equation}
Recall that
$r_+(z)=\langle (J_+-z)^{-1}e_0,e_0 \rangle$. 

\begin{lemma}\label{vbr}
Let $\fK_\bc=\l^2_+\ominus \l^2_{+,\bc}$. This is a one dimensional space, i.e.,
$\fK_\bc=\{\k_{\bc}\}$. Moreover, we can choose
\begin{equation}\label{defkap}
\k_{\bc}=(J-\bc)^{-1}(e_{-1}a(0)\sin\varphi+e_0\cos\varphi),
\end{equation}
where
\begin{equation}\label{defkap2}
\tan\varphi=\tan\varphi(\bc)=r_+(\bc),\ -\frac{\pi}{2}<\varphi\le \frac{\pi}{2},
\end{equation}
including $\varphi=\frac{\pi} 2$ if $r_+(\bc)=\infty$, that is,
 $\bc$ is a pole of this function. 
In this notations
 \begin{equation}\label{defkap3}
\|\k_\bc\|^2=\frac{r'_+(\bc)}{1+r_+(\bc)^2}=\varphi'(\bc).
 \end{equation}
 Moreover, the following two-sided estimation holds
  \begin{equation}\label{defkap4}
 \frac{\min\{a(0)^2,1\}}{(|\bc|+\|J\|)^2}\le\varphi'(\bc)\le  \frac{\max\{a(0)^2,1\}}{\dist^2(\bc,\sigma(J))}.
\end{equation}
 \end{lemma}

 \begin{proof}
 If $r_+(\bc)\not=\infty$, we have $\k_\bc=(J_+-c)^{-1}e_0\cos\varphi$.
 Otherwise $\k_{\bc}$ is collinear to the corresponding eigenvector of $J_+$. These prove
 \eqref{defkap}, \eqref{defkap2}.
 
 Further, we have
 $$
 \|\k_{\bc}\|^2=\langle (J_+-\bc)^{-2}e_0,e_0 \rangle\cos^2\varphi=\frac{r'_+(\bc)}{1+r_+(\bc)^2},
 $$
 which proves \eqref{defkap3}. We use \eqref{sh1}.
Since $\int d\Sigma=I$, we have
$$
\frac{1}{(|\bc|+\|J\|)^2}
\le R'(\bc)=\int\frac{d\Sigma}{(x-\bc)^2}\le \frac{1}{\dist^2(\bc,\sigma(J))}.
$$
Using \eqref{sh1}, we obtain
$$
\frac{1}{(|\bc|+\|J\|)^2}
\le R(\bc)\begin{bmatrix}
\frac{r_-'(\bc)}{r_-(\bc)^2}&0
\\0&\frac{r_+'(\bc)}{r_+(\bc)^2}
\end{bmatrix}R(\bc)\le \frac{1}{\dist^2(\bc,\sigma(J))},
$$
or
$$
\frac{R(\bc)^{-2}}{(|\bc|+\|J\|)^2}
\le
\begin{bmatrix}
\frac{r_-'(\bc)}{r_-(\bc)^2}&0
\\0&\frac{r_+'(\bc)}{r_+(\bc)^2}
\end{bmatrix}\le \frac{R(\bc)^{-2}}{\dist^2(\bc,\sigma(J))}.
$$
Comparing the values in the lower corner of these matrices, we get
$$
\frac{1+a_0^2r_+(\bc)^2}{(|\bc|+\|J\|)^2}
\le
r_+'(\bc)
\le \frac{1+a_0^2r_+(\bc)^2}{\dist^2(\bc,\sigma(J))}.
$$
Thus, \eqref{defkap4} is also proved.
 \end{proof}
 
 \begin{proof}[Proof of Proposition \ref{prop73}]
  Defining $\k_{\bc}$ by \eqref{defkap}, we obtain $F^*e_0=\frac{1}{\|\k_{\bc_1}\|}\k_{\bc_1}$.
In particular,
 $$
 p^{(0)}_0(0)=\langle Ae_{-1},e_0 \rangle:=\langle Je_{-1},\frac{\k_{\bc_1}}{\|\k_{\bc_1}\|} \rangle=\frac{a(0)\sin\varphi(\bc_1)}{\varphi'(\bc_1)}.
 $$
 Generally, we consider the ordered system of vectors
\begin{equation}\label{mthos0}
\k_{\bc_1},\dots,\k_{\bc_g}, e_0.
\end{equation}
This system  is linearly independent. Otherwise, there exists a nontrivial vector $\xi=\{\xi_j\}_{j=0}^g$ such that
$$
0=\langle (J-z)^{-1}(e_{-1}a_0 r_+(z)+e_0),\k_{\bc_1}\xi_0+\dots+e_0\xi_g \rangle
=\langle (J_+-z)^{-1}e_0,\k_{\bc_1}\xi_0+\dots+e_0\xi_g \rangle
$$
$$
=\frac{r_+(z)\cos\phi(\bc_1)-\sin\phi(\bc_1)}{z-\bc_1}\xi_0+\dots+\frac{r_+(z)\cos\phi(\bc_g)-\sin\phi(\bc_g)}{z-\bc_g}\xi_{g-1}+r_+(z)\xi_g
$$
$$
=\left(\frac{\cos\phi(\bc_1)}{z-\bc_1}\xi_0+\dots+\frac{\cos\phi(\bc_g)}{z-\bc_g}\xi_{g-1}+\xi_g\right)r_+(z)-
\left(\frac{\sin\phi(\bc_1)}{z-\bc_1}\xi_0+\dots+\frac{\sin\phi(\bc_g)}{z-\bc_g}\xi_{g-1}\right)
$$
That is, $r_+(z)$ is rational and the corresponding $\sigma_+$ has only a finite number of mass-points, which contradicts to the original assumption  that $a(n)>0$ for all natural $n$.

In the spectral representation the system \eqref{mthos0} corresponds to the vector-functions
$$
\frac 1{x-\bc_1}\begin{bmatrix}
a_0\sin\phi(\bc_1)\\ \cos\phi(\bc_1)
\end{bmatrix},\dots,
\frac 1{x-\bc_g}\begin{bmatrix}
a_0\sin\phi(\bc_g)\\ \cos\phi(\bc_g)
\end{bmatrix},\begin{bmatrix}
0\\1
\end{bmatrix}
$$
from $L^2_{d\Sigma}$. Jointly with its orthogonal complement they form $(2g+2)$-dimensional cyclic subspace 
\begin{equation}\label{nis1}
\frac 1{x-\bc_1}\begin{bmatrix}
\xi_{-1,0}\\ \xi_{0,0}
\end{bmatrix}+\dots+
\frac 1{x-\bc_g}\begin{bmatrix}
\xi_{-1,g-1}\\ \xi_{0,g-1}
\end{bmatrix}+\begin{bmatrix}
\xi_{-1,g}\\ \xi_{0,g}
\end{bmatrix}
\end{equation}
of the operator multiplication by $\Delta(x)$ in this space. Being ordered and orthogonalized  in an appropriate way, it generates a GMP basis in $L^2_{d\Sigma}$. The operator of multiplication by the independent variable with respect to this basis forms 
$A\in\GMP(\bC)$. Moreover, its spectral matrix measure is $d\Sigma$, that is, the spectral measure of the initial $J$.
 \end{proof}
 
 \begin{proof}[Proof of Theorem \ref{thdensks}] 
 Clearly, the eigenvalue spectral condition on $A$ corresponds to the eigenvalue spectral condition for $\Delta(A)$ of the Killip-Simon class matrices with asymptotically constant matrix-block coefficients. 
 
  We see now that the cyclic subspace \eqref{nis1} of $\Delta(A)$ for a two-sided GMP matrix $A$ represents indeed a simple two dimensional counterpart of the system \eqref{nis4}. Subsequently the matrix measure $d\Xi_2(y)$ of $\Delta(A)$ is of the form \eqref{nis3}, i.e.,
  \begin{equation*}\label{nis31}
\int \frac{d\Xi_2(y)}{y-z}:=\int\frac 1 {\Delta(x)-z}W_2^*(x)d\Sigma(x)W_2(x),
\  \ 
 W_2(x)=
\begin{bmatrix}
\frac{I_2}{\bc_1-x} &\hdots&\frac{I_2}{\bc_g-x}&I_2
\end{bmatrix}.
\end{equation*}
We have
\begin{eqnarray*}
\Xi'_2(y)&=&\sum_{\Delta(x)=y}W_2^*(x)\frac{\Sigma'(x)}{\Delta'(x)}W_2(x)\nonumber\\
&=&
W_2^*\begin{bmatrix}
\frac{\Sigma'(x_1)}{\Delta'(x_1)}& & \\
&\ddots& \\
& & \frac{\Sigma'(x_g)}{\Delta'(x_g)}
\end{bmatrix} W_2,\ W_2:=W\otimes I_2.
\end{eqnarray*}
Using $\det W_2=\det^2 W$, by a two-dimensional counterpart of \eqref{densks},   we obtain an equivalence of the Killip-Simon a.c. spectral conditions on  $A$ and $\Delta(A)$.
\end{proof}

\section{Theorem \ref{th73}: the first step in a parametrization of coefficient sequences of the Killip-Simon class}

\subsection{``Derivative" in the Jacobi flow direction}

 Let us make the block decomposition of $\Delta(A)$ in $(g+1)\times (g+1)$ blocks
 \begin{equation}\label{eq11}
 \Delta(A)=\begin{bmatrix}\ddots&\ddots&\ddots & & &\\
 &\fv^*_{-1}&\fw_{-1}&\fv_0 & & \\
  & &\fv^*_{0}&\fw_{0}&\fv_1 & \\
& &  & \ddots&\ddots\ddots
 \end{bmatrix},
 \end{equation}
 where $\fw_k$ is a self-adjoint matrix and $\fv_k$ is  a lower triangular one, i.e.,
 $$
 \fw_k=\begin{bmatrix} w^{(k)}_{0,0}&\hdots &w^{(k)}_{0,g}\\ 
 \vdots&& \vdots \\
 w^{(k)}_{g,0}& \hdots&w^{(k)}_{g,g} \end{bmatrix},
\ \ 
 \fv_k=\begin{bmatrix} v^{(k)}_{0,0}&0 &0\\ 
 \vdots&\ddots& 0 \\
 v^{(k)}_{g,0}& \hdots&v^{(k)}_{g,g} \end{bmatrix}.
$$
Due to the previous subsection and general results on Jacobi block-matrices of Killip-Simon class \cite{KSDp},
 the spectral condition \eqref{eq5} is equivalent to the boundedness of
the following KS-functional
\begin{equation}\label{eq10}
H_+(A)=\sum_{j\ge 0}h(\fv_j,\fw_j,\fv_{j+1}),
\end{equation}
where
$$
h(\fv_j,\fw_j,\fv_{j+1})=\frac 1 2\tr (\fv_j^* \fv_j
+\fw_j^2
+\fv_{j+1} \fv_{j+1}^*)
-(g+1)-\log \det\fv_j\fv_{j+1}.
$$
\begin{lemma}\label{lemjder}
Let
$$
\delta_J H_+(A)=\frac 1 2 \langle \Delta(\cJ A) e_{-1}, \Delta(\cJ A) e_{-1}\rangle-1-\log(\cJ v)_{g,g}^{(-1)}(\cJ v)_{g,g}^{(0)}.
$$
Then
\begin{equation}\label{1apr24}
H_+(A)=H_+(\cJ A)+\delta_J H_+(A).
\end{equation}
\end{lemma}

\begin{proof} Comparing $\cJ A=S^{-1}U_A^*\Delta(A)U_AS$ and $U_A^*\Delta(A)U_A$, we note that
$\delta_J H_+(A)$ adds  to $H_+(\cJ A)$ exactly that terms,  which were omitted (beause of the shift)
in the trace-like expression \eqref{eq10} for $P_+U_A^*\Delta(A)U_A$.
Further, since $U_A$  is of a block diagonal form, we have the following identities between the blocks of $U_A^*\Delta(A)U_A$
and $\Delta(A)$ itself:
\begin{alignat*}{3}
\tr\, U^*(\vp_j)\fv_j^*\fv_j U(\vp_j)&=\tr\,\fv_j^*\fv_j, \ \ \tr\, U^*(\vp_j)\fw^2_jU(\vp_j)=\tr\, \fw_j^2,\\
\tr\, U^*(\vp_j)\fv_{j+1}\fv_{j+1}^* U(\vp_j)&=\tr\, \fv_{j+1}\fv_{j+1}^*.
\end{alignat*}
Also,  all $\fv_j$ and $U^*(\vp_{j-1})\fv_j U(\vp_{j})$ are triangular matrices, and we have
$$
\prod_{l=0}^g v^{(j)}_{l,l}=\det \fv_j=\det U^*(\vp_{j-1})\fv_j U(\vp_{j}).
$$
After that, we arrive to the conclusion that the right and left hand side in \eqref{1apr24} coincide. 
\end{proof}

\subsection{``Derivative" in the GMP direction}
Let $\bTh$ be unitary in $\l^2$ such that $\mathbf\Theta:\l^2_+\to\l^2_+$. We denote by $\cK_{\bTh}=\l^2_+\ominus\bTh \l^2_+$.
For $\bTh_n=U_{A(0)}SU_{A(1)}\dots U_{A(n)}S$ we have
\begin{equation}\label{7sem0}
\cK_{\bTh_n S^{g+1}}=\cK_{\bTh_n}\oplus \bTh_n\cK_{S^{g+1}}=\cK_{S^{g+1}}\oplus S^{g+1}\cK_{S^{-(g+1)}\bTh_n S^{g+1}}.
\end{equation}
The system $\tilde e_0,\dots,\tilde e_n$, $\tilde e_j=\bTh_j e_{-1}$, forms a basis in $\cK_{\bTh_n}$, see \eqref{nis51}, and $e_0,\dots, e_g$ is the standard basis in 
$\cK_{S^{g+1}}$.  By 
$\breve e_j=\bTh_j S^{g+1}e_{-1}=\bTh_j e_{g}$, 
$j=0,\dots,n$, we denote a similar orthonormal system in $S^{g+1}\cK_{S^{-(g+1)}\bTh_n S^{g+1}}$.

\begin{remark}\label{7semrem}
We will need both, the standard shift for GMP matrices $A_1:=S^{-(g+1)}AS^{g+1}$ as well as the  shift of $A$ in the Jacobi flow direction, i.e.,
$A(1)=\cJ A$. For this reason we have to use quite complicated notations. 
Recall that $\fv^{(j)}_{km}(n)$ denotes the $(k,m)$ entry of the the block $\fv_j(n)$ in the matrix $\Delta(A(n))$, where $A(n)=\cJ^{\circ n} A$. Then
$\fv^{(j+1)}_{km}(n)$ has the same meaning with respect to the shifted GMP  matrix $A_1(n)=\cJ^{\circ n}A_1=(\cJ^{\circ n} A)_1=A(n)_1$, i.e.: $\fv^{(j+1)}_{km}(n)=\fv^{(j)}_{km}(n)_1$.
\end{remark}

\begin{lemma}
Let us define the partial sum for the functional $\tilde H_+(A)$ by
$$
\tilde H_{+,n}(A):=\sum_{m=1}^n\{\frac 1 2\langle \Delta(A(m))  e_{-1},\Delta(A(m)) e_{-1} \rangle-1-\log\fv_{gg}^{(-1)}(m)
\fv_{gg}^{(0)}(m)\}
$$
and let, as before,
$$
h(\fv_0,\fw_0,\fv_1)=\frac 1 2\tr\{{\fv_0^*\fv_0+
\fw_0^2+\fv_1\fv_1^*} \}-(g+1)-\log\det\fv_0\fv_1.
$$
Then
\begin{equation}\label{7sem2}
\tilde H_{+,n}(A)+h(\fv_0(n),\fw_0(n),\fv_1(n))
=h(\fv_0(0),\fw_0(0),\fv_1(0))
+\tilde H_{+,n}(A_1).
\end{equation}
\end{lemma}

\begin{proof}
Let $P_{\cK_{\bTh_n S^{g+1}}}$ be the orthogonal projector on $\cK_{\bTh_n S^{g+1}}$. We compute the trace of the matrix
$$
P_{\cK_{\bTh_n S^{g+1}}}\Delta(A)^2|_{\cK_{\bTh_n S^{g+1}}}
$$
with respect to the two decompositions \eqref{7sem0}.
In the first case we get
\begin{equation}\label{7sem3}
\sum_{m=0}^n\langle \Delta(A) \tilde e_m,\Delta(A)\tilde e_m \rangle=\sum_{m=0}^n\langle \Delta(A(m))  e_{-1},\Delta(A(m)) e_{-1} \rangle
\end{equation}
and
\begin{equation}\label{7sem4}
\sum_{j=0}^g\langle \Delta(A) \bTh_n e_j,\Delta(A)\bTh_n e_j\rangle=\tr\{\fv_0(n)^*\fv_0(n)+
\fw_0(n)^2+\fv_1(n)\fv_1(n)^*\}.
\end{equation}
With respect to the second orthonormal system we get
\begin{equation}\label{7sem5}
\sum_{j=0}^g\langle \Delta(A)  e_j,\Delta(A) e_j\rangle=\tr\{\fv_0^*\fv_0+
\fw_0^2+\fv_1\fv_1^*\}
\end{equation}
and
\begin{equation}\label{7sem6}
\sum_{m=0}^n\langle \Delta(A) \breve e_m,\Delta(A)\breve e_m \rangle=\sum_{m=0}^n\langle \Delta(A_1(m))  e_{-1},
\Delta(A_1(m))  e_{-1} \rangle.
\end{equation}
Thus the sum of the expressions in \eqref{7sem3} and \eqref{7sem4} is the same as in \eqref{7sem5} and \eqref{7sem6}.

Next, we claim that for an arbitrary $n$ 
\begin{equation}\label{7sem1}
v^{(0)}_{00}(0)\dots v^{(0)}_{gg}(0)\cdot\prod_{m=1}^n v^{(0)}_{gg}(m)=
\prod_{m=1}^n v^{(-1)}_{gg}(m)\cdot v^{(0)}_{00}(n)\dots v^{(0)}_{gg}(n).
\end{equation}
Recall that $v_{gg}^{(0)}(m)$ can be regarded as $v_{gg}^{(-1)}(m)_1$, see Remark \ref{7semrem}.
We prove this by induction. Since
$$
v^{(0)}_{00}(0)\dots v^{(0)}_{g g}(0)=\det \fv_0(0)=\det U^*(\vp_{-1})\fv_0 U(\vp_0)=
v^{(-1)}_{gg}(1)\cdot v^{(0)}_{00}(1)\dots v^{(0)}_{g-1g-1}(1),
$$
 we have the following recurrence relation
$$
v^{(-1)}_{gg}(1)\cdot v^{(0)}_{00}(1)\dots  v^{(0)}_{gg}(1)= v^{(0)}_{00}(0)\dots  v^{(0)}_{gg}(0) v^{(0)}_{gg}(1).
$$
That is,
\begin{align*}
v^{(0)}_{00}(0)\dots v^{(0)}_{gg}(0)\cdot v^{(0)}_{gg}(1)v^{(0)}_{gg}(2)=&
v^{(-1)}_{gg}(1)\cdot v^{(0)}_{00}(1)\dots v^{(0)}_{gg}(1) v^{(0)}_{gg}(2)\\
=& v^{(-1)}_{gg}(1) v^{(-1)}_{gg}(2)\cdot v^{(0)}_{00}(2)\dots v^{(0)}_{gg}(2),
\end{align*}
and so on... Thus, \eqref{7sem1} is proved.

In a combination of \eqref{7sem3}, \eqref{7sem4} and \eqref{7sem5}, \eqref{7sem6} with 
 \eqref{7sem1} we obtain  \eqref{7sem2}.
\end{proof}

\begin{theorem} \label{thnoch}
If $H_+(A)<\infty$, then $\tilde H_+(A)<\infty$.
If 
\begin{equation}\label{8apr24}
\tilde H_+(A)<\infty\quad  \text{and}\quad \liminf_{n\to \infty} h(\fv_0(n),\fw_0(n),\fv_1(n))=0,
\end{equation}
then the Killip-Simon functional is finite, moreover 
$H_+(A)=\tilde H_+(A)$.
\end{theorem}

\begin{proof}
 We get $\tilde H(A)<\infty$, iterating \eqref{1apr24}, morovere we obtain
 \begin{equation}\label{7apr24}
H_+(A)=\tilde H_+(A)+\lim_{m\to\infty} H_+(\cJ^{\circ m}A).
\end{equation}

If \eqref{8apr24} holds we can pass to the limit in \eqref{7sem2}
$$
\tilde H_+(A)=h(\fv_0,\fw_0,\fv_1)+\tilde H_+(S^{-(g+1)}AS^{g+1}).
$$
Iterating this identity we get that $H_+(A)$ is finite, and, in fact,
$$
\tilde H_+(A)=H_+(A)+\lim_{m\to\infty}\tilde H_+(S^{-m(g+1)}A S^{m(g+1)}).
$$
Therefore, by \eqref{7apr24},
$$
H_+(A)=
H_+(A)+\lim_{m\to\infty}\tilde H_+(S^{-m(g+1)}A S^{m(g+1)})+
\lim_{m\to\infty} H_+(\cJ^{\circ m}A).
$$
Since both limits are nonnegative, we obtain $H_+(A)=\tilde H_+(A)$.
\end{proof}

\subsection{Proof of Theorem \ref{th73}}

\begin{lemma}\label{lem51}
Let $A\in\KSA(E,\bC)$ and $A(n+1)=\cJ A(n)$, $A(0)=A$. Then \eqref{m29} and the first relation in \eqref{m30} are satisfied.
\end{lemma}

First we prove the following sublemma.
\begin{lemma}\label{lem52}
Assume that for sequences $\psi_n$ and $\tilde \psi_n$ there are sequences $\tau_n$ and $\tilde \tau_n$ such that
\begin{equation}\label{eqsubl}
\begin{bmatrix}
\tau_{n}&0\\
0
&1
\end{bmatrix}
\bo(\psi_{n})-
\bo(\tilde\psi_{n})\begin{bmatrix}
1&0\\0
&\tilde\tau_{n}
\end{bmatrix}\in \l^2_+,
\end{equation}
that is, all entries of the above matrix form $\l^2_+$-sequences. Assume in addition that there is $\eta>0$ such that
for all $n$ we have  a priori estimations
\begin{equation}\label{eqsubl1}
\cos\psi_n\ge\eta, \ \cos\tilde \psi_n\ge\eta, \ \frac 1\eta\ge \tau_n\ge\eta, \ \frac 1 \eta\ge\tilde\tau_n\ge\eta.
\end{equation}
Then $\{e^{i\psi_n}-e^{i\tilde\psi_n}\}_{n\ge 0}\in\l^2_+$.
\end{lemma}
\begin{proof}
Directly from \eqref{eqsubl} we have 
$$
\{\cos\psi_n-\cos\tilde\psi_n\}_{n\ge 0}\in\l^2_+\ \text{and}\  
\{\tau_n\cos\psi_n-\tilde \tau_n\cos\tilde\psi_n\}_{n\ge 0}\in\l^2_+.
$$
Then \eqref{eqsubl1} implies $\{\tau_n-\tilde\tau_n\}_{n\ge 0}\in\l^2_+$.
Now, we have another two conditions
$$
\{\tau_n\sin\psi_n-\sin\tilde\psi_n\}_{n\ge 0}\in\l^2_+\ \text{and}\  
\{\sin\psi_n-\tilde \tau_n\sin\tilde\psi_n\}_{n\ge 0}\in\l^2_+.
$$
Therefore,
$$
\sin{\psi_n}-\tau_n\tilde \tau_n\sin\psi_n-\tilde\tau_n(\sin\tilde\psi_n-\tau_n\sin\psi_n)
$$
belongs to $\l^2_+$, that is, $\{\sin{\psi_n}(1-\tau_n\tilde \tau_n)\}_{n\ge0}\in\l^2_+$. Thus, $(\tau_n^2-1)\sin\psi_n$ forms a $\l^2_+$-sequence, as well as $(\tau_n-1)\sin\psi_n$. Finally, since 
$$
\sin\psi_n-\sin\tilde\psi_n=\tau_n\sin\psi_n-\sin\tilde\psi_n-(\tau_n-1)\sin\psi_n,
$$
both $\{\sin\psi_n-\sin\tilde\psi_n\}_{n\ge 0}$ and $\{\cos\psi_n-\cos\tilde\psi_n\}_{n\ge 0}$ are $\l^2_+$-sequences.
\end{proof}

\begin{proof}[Proof of Lemma \ref{lem51}]
The first relation \eqref{m29} follows immediately from Lemma \ref{lemjder}. 

Let $\tilde A=\cO(A)$, see \eqref{varpidef}. We use tilde for all entries related to $\tilde A$ and $\Delta(\tilde A)$ \eqref{eq11}, respectively. The entries of $A(n)$ we denote by $\{p_j^{(k)}(n),q_j^{(k)}(n) \}$ and we use a similar notation for the entries of $\Delta(A(n))$ and $\Delta(\tilde A(n))$. Due to Definition \ref{defo},
\begin{equation}\label{eqar}
\begin{bmatrix}
v^{(0)}_{g-1,g-1}(n)& 0\\
v^{(0)}_{g,g-1}(n)& \lambda_0 p^{(0)}_g(n)
\end{bmatrix}
\bo(\phi_{g}^{(0)}(n))
=
\bo(\phi_{g}^{(-1)}(n))\begin{bmatrix}
\lambda_0\tilde p^{(0)}_g(n)
& 0\\
\tilde w^{(0)}_{0,g}(n)& \tilde v^{(1)}_{0,0}(n)
\end{bmatrix}.
\end{equation}
Applying  Lemma \ref{lemjder} to the matrix $A$, we obtain
$$
\{\lambda_0 p^{(0)}_g(n)-1\}_{n\ge 0}\in \l^2_+,\quad \{v^{(0)}_{g,g-1}(n)\}_{n\ge 0}\in \l^2_+.
$$
Similarly for the entries related to $\tilde A$ we have
$
\{\lambda_0 \tilde p^{(0)}_g(n)-1\}_{n\ge 0}\in \l^2_+,\quad  \{\tilde w^{(0)}_{0,g}(n)\}_{n\ge 0}\in \l^2_+.
$
 Thus, we can apply Lemma \ref{lem52} with respect to \eqref{eqar}. We get $\{\sin\phi^{(-1)}(n)-\sin\phi^{(0)}(n)\}$ belongs to $\l^2_+$. That is, $\{p^{(-1)}_{g-1}(n)-p^{(0)}_{g-1}(n)\}_{n\ge 0}\in\l^2_+$.
 
 Using \eqref{def3}, \eqref{def2}, we get similar relations for all others $j$'s.  
 Using \eqref{def1}, we prove the second part of \eqref{m29}.
\end{proof}

\begin{proof}[Proof of Theorem \ref{th73}]
Lemma \ref{lem52} implies that 
$
(v^{(-1)}_{g-1,g-1}(n)-1)\sin\phi^{(-1)}_g(n)
$
form an $\l^2_+$-sequence, or, equivalently, see \eqref{deflad},
\begin{equation}\label{lambdal2}
\{(\Lambda^\#_{-1,g}(n)-\lambda_g) p^{(-1)}_{g-1}(n)\}_{n\ge 0}\in \l^2_+.
\end{equation}
Since $p^{(-1)}_{g-1}(n)$ may approach to zero, it does not imply yet that $\{\Lambda^\#_{-1,g}(n)-\lambda_g\}$ belongs to $\l ^2_+$. 
Let us show that
\begin{equation}\label{lambdal21}
\{(\Lambda^\#_{-1,g}(n)-\lambda_g) q^{(-1)}_{g-1}(n)\}_{n\ge 0}\in \l^2_+.
\end{equation}
Since $\inf_{n}\left((q^{(-1)}_{g-1}(n))^2+(p^{(-1)}_{g-1}(n))^2\right)>0$, both \eqref{lambdal2} and \eqref{lambdal21} give us
\eqref{m31} for $m=g$.

To this end, we note that 
\begin{equation}\label{lambdal22}
\Lambda^\#_{-1,g}(n+1)=\frac{\cos\phi^{(-1)}_g(n)}{\cos\phi^{(-2)}_g(n)}\Lambda^\#_{-1,g}(n).
\end{equation}
 Indeed, by definition
of the Jacobi flow
$$
U(\vp_{-2}(n))\begin{bmatrix} v^{(-2)}_{g,g}& & & \\
*&v^{(-1)}_{0,0}& &\\
 *&* & \ddots &\\
 *&* &* &v^{(-1)}_{g-1,g-1}
\end{bmatrix}(n+1)=\fv_{-1}(n) U(\vp_{-1}(n))
$$
the second from below entry in the last column in this matrix identity means exactly \eqref{lambdal22}.
Therefore, by Lemma \ref{lem52}, we get 
\begin{equation}\label{lambdal23}
\{\Lambda^\#_{-1,g}(n+1)-\Lambda^\#_{-1,g}(n)\}_{n\ge 0}\in \l^2_+. 
\end{equation}
Now,
by \eqref{jfex}
\begin{equation*}
(\Lambda^\#_{-1,g}(n)-\lambda_g) p^{(-1)}_{g-1}(n)=
-(\Lambda^\#_{-1,g}(n)-\lambda_g) q^{(-1)}_{g-1}(n+1)
\frac{p_g^{(-1)}(n) \|d_{g-1}\vp_{-1}(n)\|}{
\|\vp_{-1}(n)\|}.
\end{equation*}
In combination with \eqref{lambdal23} we have  \eqref{lambdal21}, and therefore \eqref{m31} for $k=g$.

The same arguments with respect to $\cO^k A$, $k=1,...,g-1$, in a combination with \eqref{shiftjflow}, give \eqref{m31} for all other $k$.

Now we start with the conditions \eqref{m29}-\eqref{m31}. Let us show that they imply the second relation in \eqref{8apr24}. Due to compactness we can choose convergent subsequences 
\begin{equation}\label{noch}
\vp_0(n_k)\to\oc\vp\quad\text{and}\quad \vq_0(n_k)\to\oc\vq.
\end{equation}
Let $\oc A$ be the periodic GMP matrix generated by $(\oc\vp, \oc\vq)$. Passing to the limit in  \eqref{m30} and \eqref{m31}
along the subsequence $\{n_k\}$ we get  the isospectral conditions
\eqref{iso101}, that is,  $\oc A\in A(E,\bC)$. By the magic formula, $\oc\fv=I$, $\oc\fw=0$, where $\oc\fv,\oc\fw$ are blocks of $\Delta(\oc A)$. Recall that blocks of $\Delta(A)$ are formed by the coefficients of two consecutive blocks of $A$, see Lemma \ref{lem:gsmpEntries}. Therefore, due to \eqref{noch} and  \eqref{m29}, we obtain
$$
\lim_{k\to\infty}\fw_0(n_k)=\oc\fw,\quad \lim_{k\to\infty}\fv_0(n_k)=\lim_{k\to\infty}\fv_1(n_k)=\oc\fv.
$$
Thus, $\lim_{k\to\infty}h(\fv_0(n_k),\fw_0(n_k),\fv_1(n_k))=0$.

To show the first relation in \eqref{8apr24}, we evaluate the entries of $\Delta(A)e_{-1}$.
Let $h(\bc_k)=(\bc_k-A)^{-1}e_{-1}$.  In notations of Lemma \ref{lem:gsmpEntries} we have
\begin{equation*}
(h^{(-1)}(\bc_{k}))_{l}=\frac{(f^{(-1)})_l}{(f^{(-1)})_{k-1}}(h^{(-1)})_{k-1}
=\frac{\tilde\rho^{(0)}_{k-1}}{\Lambda^\#_{-1,k}}\frac{(\bp^{(-1)}_{k-1})^*\fj}{\bc_{l+1}-\bc_k}
\prod_{j=k}^{l-1}\fa(\bc_k,\bc_{j+1};\bp_j^{(-1)})
\bp_l^{(-1)},
\end{equation*}
for $k\le l$ and  $(h^{(-1)}(\bc_{l+1}))_{l}=-\frac{\tilde\rho^{(0)}_{l}}{\Lambda^\#_{-1,l+1}}$,
 where
\begin{equation}\label{altda1}
\begin{bmatrix}
\tilde \pi ^{(0)}_{k-1}\\
\tilde \rho^{(0)}_{k-1}
\end{bmatrix}=\prod_{j=0}^{k-2}\fa(\bc_k,\bc_{j+1};\bp^{(0)}_j)\bp_{k-1}^{(0)}, 
\
\begin{bmatrix}
\pi ^{(-1)}_{k-1}&
 \rho^{(-1)}_{k-1}
\end{bmatrix}=
-(\bp_{k-1}^{(-1)})^*\fj
\prod_{j=k}^{g-1}\fa(\bc_k,\bc_{j+1};\bp^{(-1)}_j)\fj
\end{equation}

We note that  due to the uniform estimations \eqref{altdef} and \eqref{oppalt1}, conditions \eqref{m29} 
and \eqref{m31} imply that
\begin{equation}\label{11m1}
\left\{\frac{\lambda_k}{\Lambda^\#_{-1,k}(n)}-1\right\}_{n\ge 0}\in\l_+^2.
\end{equation}
For $l<g$, by definition we have
$$
w_{l,g}^{(-1)}(n)=\lambda_0 p_g^{(-1)}(n)q_l^{(-1)}(n)+\lambda_{l+1}(h^{(-1)}(\bc_{l+1}))_{l}(n)+
\sum_{k=1}^{l}\lambda_k (h^{(-1)}(\bc_{k}))_{l}(n).
$$
We substitute in this expression $\tilde\rho^{(0)}_{k-1}$ from \eqref{altda1}. Then \eqref{11m1} and the first relation in \eqref{m30} imply that $\{w_{l,g}^{(-1)}(n)\}$ differs from the sequence $\{\tilde w_{l,g}^{(-1)}(n)\}$, see below, by a $\l_+^2$-sequence; here
\begin{align}\label{11m2}
 \begin{bmatrix}*\\ \tilde w_{l,g}^{(-1)}(n)\end{bmatrix}:=\bp_l^{(-1)}(n)-
\prod_{j=0}^{l-1}\fa(\bc_{l+1},\bc_{j+1};\bp^{(0)}_j(n))\bp_{l}^{(0)}(n)+
\\
\sum_{k=1}^{l}\prod_{j=0}^{k-2}\fa(\bc_k,\bc_{j+1};\bp^{(0)}_j(n))\bp_{k-1}^{(0)}(n)\frac{(\bp^{(-1)}_{k-1}(n))^*\fj}
{\bc_{l+1}-\bc_k}
\prod_{j=k}^{l-1}\fa(\bc_k,\bc_{j+1};\bp_j^{(-1)}(n))
\bp_l^{(-1)}(n)\nonumber
\end{align}
Using \eqref{m29} once again, we can substitute in the last expression all $\bp_j^{(0)}(n)$ by $\bp_j^{(-1)}(n)$. After that, we note that the product
$$
\prod_{j=0}^{k-2}\fa(\bc_k,\bc_{j+1};\bp^{(-1)}_j(n))\bp_{k-1}^{(-1)}(n){(\bp^{(-1)}_{k-1}(n))^*\fj}
\prod_{j=k}^{l-1}\fa(\bc_k,\bc_{j+1};\bp_j^{(-1)}(n))
$$
is the residue of the matrix function 
$\prod_{j=0}^{l-1}\fa(z,\bc_{j+1};\bp^{(-1)}_j(n))$ at $\bc_k$.
It remains to use  the identity 
\begin{equation}\label{11m7}
\sum_{k=1}^{l}\frac{\Res_{\bc_k} \prod_{j=0}^{l-1}\fa(z,\bc_{j+1};\bp^{(-1)}_j(n))}{z-\bc_k}
=\prod_{j=0}^{l-1}\fa(z,\bc_{j+1};\bp^{(-1)}_j(n))-I
\end{equation}
evaluated at $z=\bc_{l+1}$ and we get cancellation of the first and second line in \eqref{11m2}, i.e., 
$\{\tilde w^{(-1)}_{l,g}(n)\}\in\l^2_+$. As  result,  we obtain 
$\{w^{(-1)}_{l,g}(n)\}\in\l^2_+$ for all $0\le l<g $. 

The diagonal entry require a little bit special consideration. To prove that $\{w^{(-1)}_{g,g}(n)\}$ belongs to $\l^2_+$ we use 
\begin{equation*}\label{altd2bis}
\langle \Delta(A) e_{-1},e_{-1}\rangle=\lambda_0 p^{(-1)}_g q^{(-1)}_g+\bc_0-\sum_{k=1}^g\lambda_k\frac{\pi^{(-1)}_{k-1}\tilde \rho^{(0)}_{k-1}}{p_g^{(-1)}\Lambda^\#_{-1,k}},
\end{equation*}
 the second relation in \eqref{m30}, and, instead of \eqref{11m7}, a more involved identity \eqref{alternativqg} shown in Lemma \ref{lem:c0formula}. Similarly, one can prove that $\{v_{g,g}^{(-1)}(n)-1\}$ and $\{v_{g,l}^{(0)}(n)\}$, for $0\leq l< g$, form $\l^2_+$ sequences. Thus, by Theorem \ref{thnoch}, $H_+(A)=\tilde H_+(A)<\infty$ and, due to the matrix version of the Killip-Simon theorem  $A\in\KSA(E,\bC)$.

\end{proof}

\section{Proof of the main Theorem \ref{mainhy}}\label{Section7}
 
 \subsection{From GMP to Jacobi}

Assume that $A\in\GMP(\bC)$. Let $A(n)=\cJ^{\circ n} A$. Recall that the coefficients of the Jacobi matrix $J=\cF A$ are
given by \eqref{coefflow} and $\l^2$ properties of the coefficients $\{\vp_{\pm 1}(n),\vq_{\pm1}(n),\vp_0(n),\vq_0(n)\}$ are given in Theorem \ref{th73}. We consider the isospectral surface $\is$ given by \eqref{iso101}, see \eqref{explrhk},
with the identification $(p_j,q_j)\equiv(-p_j,-q_j)$, $j=0,\dots,g-1$.
Note that this is a $g$ dimensional torus, which we can parametrize by $\alpha\in\bbR^g/\bbZ^g	$ according to 
Theorem \ref{thm:multbyzsmp}. Moreover, by  statement (d) of  Moser-Uhlenbeck Theorem \cite[Theorem 4.7]{MUM}, for the given manifold
\begin{equation}\label{grad}
\sup_{\{\vbp\}\in\is} \|(T(\vbp)^*T(\vbp))^{-1}\|<\infty,\quad
T(\vbp)=\begin{bmatrix}
\frac{\partial\Lambda_1}{\partial p_0}&\hdots &\frac{\partial\Lambda_g}{\partial p_{0}}\\
\frac{\partial\Lambda_1}{\partial q_0}&\hdots &\frac{\partial\Lambda_g}{\partial q_{0}}\\
\vdots&\hdots&\vdots\\
\frac{\partial\Lambda_1}{\partial p_{q-1}}&\hdots &\frac{\partial\Lambda_g}{\partial p_{g-1}}\\
\frac{\partial\Lambda_1}{\partial q_{q-1}}&\hdots &\frac{\partial\Lambda_g}{\partial q_{g-1}}
\end{bmatrix}.
\end{equation}
Note that evidently
\begin{equation*}\label{grad0}
0<\inf_{\{\vbp\}\in\is} \|(T(\vbp)^*T(\vbp))^{-1}\|^{-1}\le \sup_{\{\vbp\}\in\is} \|T(\vbp)^*T(\vbp)\|<\infty.
\end{equation*}

We define a periodic GMP matrix $A(\alpha_n)$ generated by $\{\oc\vbp(\alpha_n)\}\in \is$ such that
\begin{equation}\label{estmain1}
\dist(\vbp_0(n),\is)=\dist(\vbp_0(n),\oc\vbp(\alpha_n)).
\end{equation}
Using the standard Lagrange multipliers method, we can estimate the distance from $\vbp_0(n)$ to the isospectral set in terms of 
$\sup\|(T(\vbp)^*T(\vbp))^{-1}\|$. Then, by \eqref{m30}, \eqref{m31} and \eqref{grad}, we have
\begin{equation}\label{estmain}
\sum_{n=0}^\infty\dist^2(\vbp_0(n),\oc\vbp(\alpha_n))<\infty
\end{equation}
and also, see \eqref{coefflow},
\begin{equation}\label{estmain2}
a(n)^2-\cA(\alpha_n)\in\l^2,\quad b(n)-\cB(\alpha_n)\in\l^2.
\end{equation}

On the other hand, by \eqref{m29}-\eqref{m31} and the uniform smoothness of the Jacobi flow transform
\eqref{jfex}, \eqref{jfex1}
\begin{equation*}
\dist(\vbp_0(n+1),\oc\vbp(\alpha_n-\mu))
\le C(E,J)\{
\dist(\vbp_0(n),\oc\vbp(\alpha_n))+
\dist(\vbp_0(n),\vbp_1(n))\}.
\end{equation*}
That is,
\begin{eqnarray*}
\dist(\oc\vbp(\alpha_{n+1}),\oc\vbp(\alpha_n-\mu))
&\le& C(E,J)\{
\dist(\vbp_0(n),\oc\vbp(\alpha_{n}))
+\dist(\vbp_0(n),\vbp_1(n))\}\\
&+&\dist(\vbp_0(n+1),\oc\vbp(\alpha_{n+1})).
\end{eqnarray*}
Since
$$
\|\alpha-\beta\|\le C_1(E)
\dist(\oc\vbp(\alpha),\oc\vbp(\beta)),
$$
\eqref{estmain} and \eqref{m29} imply
$$
\sum_{n=0}^\infty\|\epsilon_\alpha(n)\|^2<\infty, \quad \text{where}\quad \epsilon_\alpha(n):=\alpha_{n+1}-(\alpha_n-\mu).
$$
In combination with \eqref{estmain2}, we obtain \eqref{132}.

\begin{remark}\label{rem71}
Of course in this proof it is not necessary to choose $\alpha_n$ as the best approximation to $\vbp_0(n)$,
see \eqref{estmain1}. It is enough to have this distance under an appropriate control. This explains a certain ambiguity in the representation \eqref{132}. 
\end{remark}

\subsection{From Jacobi to GMP}\label{subs72}

In this section our goal is to estimate
$p_j(n)-\oc p_j(\alpha_n)$ and $q_j(n)-\oc q_j(\alpha_n)$
by means of the related distances $\dist((S^{-n}JS^n)_+, J(E))<\infty$ and then apply Theorem \ref{th73}.
Therefore first of all we prove the following lemma. 
Note that the relation \eqref{distalpn} below evidently implies a word-by-word counterpart of \eqref{opp} in DKST, see Remark \ref{rem7}.

\begin{lemma}\label{lem72}
Let $J$ be  of the form \eqref{132}. Then
\begin{equation}\label{distalpn}
\sum_{n=0}^\infty \dist^2_\eta((S^{-n}JS^n)_+,J(\alpha_n)_+)<\infty, \quad 
\alpha_n=\sum_{k=0}^n\epsilon_{\alpha}(k)-\mu n.
\end{equation}
\end{lemma}

\begin{proof}
We have
$$
|b(k+n)-\cB(\alpha_n-\mu k)|\le |\e_b(k+n)|+C_1(E)
\|\sum_{j=n+1}^{n+k}\epsilon_\alpha(j)
\|,
$$
where 
$
C_1(E)=\sup_{\alpha\in\bbR^g/\bbZ^g}\|\text{\rm grad}\, \cB(\alpha)\|.
$
For $\eta<1$, we have
$$
\sum_{n\ge 0} \left(\sum_{k\ge 1}
\|\sum_{j=n+1}^{n+k}\epsilon_\alpha(j)
\|^2\eta^{2k}\right)\le
\sum_{n\ge 0} \sum_{k\ge 1}\sum_{j=n+1}^{n+k}\|\epsilon_\alpha(j)\|^2k\eta^{2k}
=\sum_{k\ge 1} \sum_{n\ge 0}\sum_{j=n+1}^{n+k}\|\epsilon_\alpha(j)\|^2k\eta^{2k}
$$
$$
\le
\sum_{k\ge 1} k\eta^{2k}\sum^\infty_{j\ge 1}k \|\epsilon_\alpha(j)\|^2
\le 
\sum^\infty_{j\ge 1}\|\epsilon_\alpha(j)\|^2\cdot
\sum_{k\ge 1} k^2\eta^{2k}.
$$
Making a similar estimation for $|a(k+n)^2-\cA(\alpha_n-\mu k)|$ we obtain \eqref{distalpn}. 
\end{proof}

 Before to proceed we  make the following important for us remark.
 
 \begin{remark}\label{inirem74}
 Note that  $\fK_z$, see Lemma \ref{vbr}, is well define for all $z\in \bbC\setminus \sigma(J)$. That is, in fact, we have a Hermitian analytic vector bundle in this domain.  Its fundamental characteristic, the so-called curvature, is of the form 
 $\Delta\log\langle \k_z,\k_z \rangle$, see  e.g. \cite{CD}. Being restricted on the real axis, it represents the Schwarzian derivative of $r_+(z)$. Our further  considerations are based on estimations of related expressions and involve derivatives of exactly this level, see Lemma \ref{le75} below.
 We can conjecture that a certain Hermitian analytic vector bundle model, which generalized the model described in Section 2, is possible for operators of Killip-Simon class. Under more restrictive assumptions, when the absolutely continuous part of the spectral measure satisfies the Szeg\"o condition and positions of the eigenvalues outside $E$ obey the Blaschke condition, such model does exist. This is the  so-called scattering model for the given operator \cite{PY, VY, PVY}.
 \end{remark}
 
 \begin{lemma}\label{le75}
 Let $\oc J\in J(E)$. In the notations of Lemma \ref{vbr}, 
 \begin{equation}\label{mlkap1}
\langle (J-\oc J)\k_{\bc},\oc\k_{\bc} \rangle=\sin(\oc\varphi(\bc)-\varphi(\bc)).
\end{equation}
Consequently, there exists $C=C(\sigma(J),\bc)<\infty$ such that
\begin{equation}\label{mlkap2}
|\sin(\oc\varphi(\bc)-\varphi(\bc))|\le C\dist_\eta(J_+,\oc J_+)
\end{equation}
and simultaneously for the  derivatives
\begin{equation}\label{mlkap3}
|(\oc\varphi)^{(m)}(\bc)-\varphi^{(m)}(\bc)|\le C\dist_\eta(J_+,\oc J_+)
\end{equation}
for $m=1,2,3$ and $\eta>|b(\bc)|$.
 \end{lemma}
 \begin{proof}
 We have
 $$
 \langle (J-\oc J)\k_{\bc},\oc\k_{\bc} \rangle=
 \langle (J-\bc)\k_{\bc},\oc\k_{\bc} \rangle-
 \langle (\oc J-\bc)\k_{\bc},\oc\k_{\bc} \rangle.
 $$
 We simplify the first term
 $$
  \langle e_{-1}a(0)\sin\varphi+e_0 \cos\varphi,\oc\k_{\bc} \rangle=
  \frac{1}{\oc a(0)}\langle (\oc J-c)e_{-1},\oc\k_{\bc} \rangle  \cos\varphi=\sin\oc\varphi\cos\varphi.
 $$
 Thus,
 $
 \langle (J-\oc J)\k_{\bc},\oc\k_{\bc} \rangle=\sin\oc\varphi\cos\varphi
 -\sin\varphi\cos\oc \varphi
 $
 and \eqref{mlkap1} is proved. 
 
  The upper estimation in \eqref{defkap4}
in combination with \eqref{mlkap1},
 \eqref{defkap3} implies
 \begin{equation}\label{mlkap7}
|\sin(\oc\varphi(\bc)-\varphi(\bc))|\le \sqrt{\varphi'(\bc)(\oc \varphi)'(\bc)}\frac{\| (J-\oc J)\oc\k_{\bc} \|}{\|\oc \k_{\bc}\|}
\le C\frac{\| (J-\oc J)\oc\k_{\bc} \|}{\|\oc \k_{\bc}\|}.
\end{equation}
Now, the vector $\frac 1{\|\oc\k_\bc\|}\oc\k_\bc$ in the functional model for $\oc J=J(\alpha)$ corresponds to the normalized reproducing kernel $\frac{1}{\| k^\alpha_{\zeta_\bc}\|} k^\alpha_{\zeta_\bc}$,
where $\zeta_\bc\in\bbD$ is such that $\fz(\zeta_\bc)=\bc$. 
 The components of this vector were estimated in \eqref{estF}. Thus,
$$
\frac 1{\|\oc \k_{\bc}\|}{\| (J-\oc J)\oc\k_{\bc} \|}\le C(E)\dist_\eta(J_+,\oc J_+), \quad |b(\bc)|<\eta<1,
$$
and \eqref{mlkap7} implies \eqref{mlkap2}.

To get \eqref{mlkap3} we differentiate \eqref{mlkap1} with respect to $\bc$
\begin{equation}\label{firstde}
\cos(\oc\varphi(\bc)-\varphi(\bc))((\oc\varphi)'(\bc)-\varphi'(\bc))=\langle (J-\oc J)\k'_\bc,\oc\k_\bc \rangle+
\langle (J-\oc J)\k_\bc,(\oc\k_\bc)' \rangle.
\end{equation}
Since $\sin(\oc\varphi(\bc)-\varphi(\bc))$ was estimated from above, we have a uniform estimation for 
$|\cos(\oc\varphi(\bc)-\varphi(\bc))|$ from below. Using \eqref{defkap}, we evaluate $\k'_\bc$. Based on its explicit form and the estimation for $\varphi'_\bc$, we obtain that $\|\k'_{\bc}\|$ is also bounded by the distance from $\bc$ to $\sigma(J)$. Evidently, the coefficients of $(\oc \k_{\bc})'$ also satisfies \eqref{estF}. Thus,
$$
|(\oc\varphi)'(\bc)-\varphi'(\bc)|=\frac{\|(J-\oc J)\oc\k_c \|\|\k'_\bc\|+
\|(J-\oc J)(\oc\k_c)' \|\|\k_\bc\|}{|\cos(\oc\varphi(\bc)-\varphi(c))|}
$$
implies \eqref{mlkap3}. Taking the second and third derivatives in \eqref{firstde}, we obtain \eqref{mlkap3} for $m=2,3$.
 \end{proof}

\begin{corollary} If $J$ is of the form \eqref{132} and $A(n)=\cF^{-1}(S^{-n} J S^n)$,
 then
 \begin{equation}\label{zzz}
\sum_{n=0}^\infty|p^{(0)}_0(n)-p_0(\alpha_n)|^2<\infty, \quad \alpha_n= \sum_{k=0}^n\epsilon_\alpha(k)-\mu n.
\end{equation}

\end{corollary}
\begin{proof}
By \eqref{mlkap2}, \eqref{mlkap3} we can estimate the difference
$$
p^{(0)}_0(n)-p_0(\alpha_n)=\frac{a(n)\sin\varphi(\bc_1)}{\varphi'(\bc_1)}-
\frac{\oc a(0)\sin\oc\varphi(\bc_1)}{(\oc\varphi)'(\bc_1)},\quad \oc J=J(\alpha_n),
$$
by means of $\dist((S^{-n}J S^n)_+,J(\alpha_n)_+)$. Due to \eqref{distalpn}, we have \eqref{zzz}.
\end{proof}

\begin{proof}[Finishing the proof of Theorem \ref{mainhy}]
It remains to show that \eqref{132} imply \eqref{m29}-\eqref{m31}.
Similarly to \eqref{mthos0}, consider the ordered system of vectors
\begin{equation}\label{mthos}
e_{-1},\k_{\bc_1},\dots,\k_{\bc_g},e_0,\k'_{\bc_1},\dots,\k'_{\bc_g},e_1.
\end{equation}
Let us point out that the orthogonalization of the system
\begin{equation}\label{mthos1}
e_{-1},\oc\k_{\bc_1},\dots,\oc\k_{\bc_g},e_0,(\oc\k_{\bc_1})',\dots,(\oc \k_{\bc_g})',e_1
\end{equation}
leads to the family $\{\ff^\alpha_j\}_{j=-1}^{2g+2}$, see \eqref{smpbase}, where $\oc J=J(\alpha)$.

To evaluate the Gram-Schmidt matrix of the system \eqref{mthos} we use
$$
\langle \k_{\bc_j},\k_{\bc_m} \rangle=\frac{r_+(\bc_j)-r_+(\bc_m)}{\bc_j-\bc_m}\cos\varphi(\bc_j)\cos\varphi(\bc_m)=
\frac{\sin(\varphi(\bc_j)-\varphi(\bc_m))}{\bc_j-\bc_m}.
$$
Therefore,
$$
\langle \k'_{\bc_j},\k_{\bc_m} \rangle=\frac{\cos(\varphi(\bc_j)-\varphi(\bc_m))}{\bc_j-\bc_m}\varphi'(\bc_j)-
\frac{\sin(\varphi(\bc_j)-\varphi(\bc_m))}{(\bc_j-\bc_m)^2}
$$
and
\begin{eqnarray*}
\langle \k'_{\bc_j},\k'_{\bc_m} \rangle&=&
\frac{\cos(\varphi(\bc_j)-\varphi(\bc_m))}{(\bc_j-\bc_m)^2}(\varphi'(\bc_j)+\varphi'(\bc_m))
\\
&+&
\frac{\sin(\varphi(\bc_j)-\varphi(\bc_m))}{\bc_j-\bc_m}\varphi'(\bc_j)\varphi'(\bc_m)
-2\frac{\sin(\varphi(\bc_j)-\varphi(\bc_m))}{(\bc_j-\bc_m)^3},\ j\not=m.
\end{eqnarray*}

Having uniform estimations from below for all Gram-Schmidt determinants of the system \eqref{mthos1},
from  \eqref{mlkap2}, \eqref{mlkap3}, similarly to \eqref{zzz}, we obtain
$$
\sum_{n=0}^\infty|p^{(m)}_j(n)-p_j(\alpha_n)|^2<\infty,
\ \sum_{n=0}^\infty|q^{(m)}_j(n)-q_j(\alpha_n)|^2<\infty,\quad m=-1,0,1,\ j=0,\dots,g.
$$
 This implies \eqref{m29}-\eqref{m31}, in particular,
$
\sum_{n=0}^\infty|p^{(\pm 1)}_j(n)-p^{(0)}_j(n)|^2<\infty,$  $j=0,...,g-1$.

\end{proof}

\section{Appendix: one sided GMP matrices}
We describe interrelations between one sided GMP  and Jacobi matrices  given by, see \eqref{ijf4},
\begin{equation}\label{app1}
r_-(z)=\langle (J_--z)^{-1} e_{-1}, e_{-1} \rangle=\langle (A_--z)^{-1} e_{-1}, e_{-1} \rangle=\int\frac{d\sigma_-(x)}{x-z}
\end{equation}
Let us point out that both matrices have the same cyclic vector, and they are related by a common spectral measure $d\sigma_-$.
It is assumed that $\bc_1,\dots,\bc_g$ do not belong to the (closed) support of this measure.

Thus, as soon as $A_-$ is given we can construct $J_-$ in the usual way, making basis of orthonormal polynomials.
In the opposite direction we will construct an orthonormal basis of rational functions in $L^2_{d\sigma_-}$. The matrix of the multiplication operator by the independent variable in this basis  $\{\tau_{k}(x)\}_{k\le -1}$  is $A_-$.

\begin{definition}
To the given $d\sigma_-$ we associate the orthonormal system
\begin{equation}\label{app2}
\begin{bmatrix}\tau_{-1}(x)& \tau_{-2}(x)&\hdots& \tau_{-g-1}(x)\end{bmatrix}
=\begin{bmatrix}1&\frac{1}{\bc_g-x}&\hdots&\frac{1}{\bc_1-x}
\end{bmatrix} L,
\end{equation}
where $L$ is the upper triangular matrix
$$
 L=\begin{bmatrix}
\ell^{(-1)}_{0}&\ell^{(-2)}_0&\dots&\ell^{(-g-1)}_0\\
&\ell^{(-2)}_1&\dots&\ell^{(-g-1)}_1\\
&&\ddots& \vdots\\
&& & \ell^{(-g-1)}_g
\end{bmatrix}, \quad \ell^{(-k-1)}_k>0,
$$
such that $\int \tau_k(x)\tau_j(x)d\sigma_-(x)=\delta_{k,j}$.
\end{definition}

In other words, if $D$ is the Gram-Schmidt matrix of the given system
$$
D=
\begin{bmatrix}
1&r_-(\bc_g)&\dots& r_-(\bc_1)\\
r_-(\bc_g)&r_-'(\bc_g)&\dots&\frac{ r_-(\bc_1)-r_-(\bc_g)}{\bc_1-\bc_g}\\
\vdots&\vdots &\ddots& \vdots\\
r_-(\bc_1)& \frac{ r_-(\bc_1)-r_-(\bc_g)}{\bc_1-\bc_g}&\dots& r_-'(\bc_1)
\end{bmatrix},\ \frac{r_-(\bc_j)-r_{-}(\bc_k)}{\bc_j-\bc_k}=\int\frac{d\sigma_-(x)}{(x-\bc_j)(x-\bc_k)},
$$
then $L$ is defined via the upper-lower triangular factorization of $D^{-1}$
\begin{equation}\label{appn1}
L^* D L=I \quad\text{or}\quad D^{-1}=L L^*.
\end{equation}

\begin{lemma}\label{lemapp1}
The multiplication operator with respect to the orthonormal system \eqref{app2} represents the $B$-block of a GMP matrix (see \eqref{n2}, \eqref{n3}), that is,
for the matrix $B_{-1}$ given by
\begin{equation}\label{app3}
B^{(-1)}_{jk}:=
\int x\tau_j(x)\tau_k(x) d\sigma_-(x) 
\end{equation}
its lower triangular part {\em (including the main diagonal)} is of the form
\begin{equation}\label{app4}
B^+=(\vec{\zm}(\vec{\ell})^*)^++\hat \bC,
\end{equation}
where
$$
\vec\ell:=\begin{bmatrix}
\ell_0^{(-1)}\\\ell_0^{(-2)}\\ \dots\\\ell_0^{(-g-1)}
\end{bmatrix}, \vec{\zm}=\begin{bmatrix}
\int \tau_{-1}(x) xd\sigma_-(x)\\
\int \tau_{-2}(x) xd\sigma_-(x)\\
\vdots\\
\int \tau_{-g-1}(x) xd\sigma_-(x)\\
\end{bmatrix}, \quad
\hat\bC=\begin{bmatrix}
0& & &\\
&\bc_g & &\\
& & \ddots& \\
& & & \bc_1
\end{bmatrix}.
$$
\end{lemma}

\begin{proof}
Note that
\begin{equation}\label{app0}
x\begin{bmatrix}1&\frac{1}{\bc_g-x}&\hdots&\frac{1}{\bc_1-x}
\end{bmatrix}=
\begin{bmatrix}x&-1&\hdots&-1
\end{bmatrix}
+\begin{bmatrix}1&\frac{1}{\bc_g-x}&\hdots&\frac{1}{\bc_1-x}
\end{bmatrix}\tilde\bC.
\end{equation}
We substitute this and \eqref{app2} in \eqref{app3}. Since constants  are orthogonal to $\tau_{k}(x)$ for all $k$, except for $k=-1$, we obtain
$$
B=
\vec{\zm}
\begin{bmatrix}
\ell_0^{(-1)}&\ell_0^{(-2)}&\dots&\ell_0^{(-g-1)}
\end{bmatrix}+
\delta_0
\begin{bmatrix}
0&*&\dots&*
\end{bmatrix}+
L^*D\hat \bC L.
$$
By \eqref{appn1} $L^*D\hat\bC L=L^{-1}\hat\bC L$. Since  $(L^{-1}\hat \bC L)^+=\hat \bC$, we have \eqref{app4}.
\end{proof}

Now, we consider the system \eqref{app2} as a \textit{cyclic} subspace for the multiplication by 
\begin{equation}\label{addm1}
\Delta(x)=\lambda_0 x+\bc_0+\sum_{k=1}^g\frac{\lambda_k}{\bc_k-x}, \quad \lambda_j>0,\ j=0,\dots,g,
\end{equation}
in $L^2_{d\sigma_-}$. We inductively define the orthonormal system 
\begin{equation}\label{app5}
\tau_{-m(g+1)-k-1}(x)=\ell^{(-m(g+1)-k-1)}_{m(g+1)+k}\Delta(x)^m\tau_{-k-1}(x)+\dots, \quad m\ge 1,\ 0\le k\le g.
\end{equation}

\begin{lemma}
The multiplication by $\Delta(x)$ with respect to the system \eqref{app5} is a $(2g+3)$-diagonal matrix, or a $(g+1)$-block diagonal Jacobi matrix
\begin{equation}\label{app6}
\cG_-=\begin{bmatrix}
\fw_{-1}&\fv^*_{-1}& & \\
 \fv_{-1}&\fw_{-2}&\fv^*_{-2}& \\
  &\ddots&\ddots&\ddots
\end{bmatrix},
\end{equation}
with upper-triangular $\fv_{k}$'s, having positive diagonal entries.
\end{lemma}
\begin{proof}
The relation $\cG_- e_m=\frac{\ell^{(m)}_{-m+1}}{\ell_{-m+g}^{(m-g-1)}}e_{m-g-1}+\dots$ follows from the definition \eqref{app5}. That is, $\fv_m$ is upper-triangular. The fact that the operator is self-adjoint implies that all $\fv^*_m$ are lower-triangular matrices.

Note that corresponding to this $\cG_-$ matrix measure $d\Xi_-(y)$ is of the form \eqref{nis3}.
\end{proof}
As it was claimed, we define $A_-=A_-(J_-)$ as the matrix of the operator multiplication by $x$  in the basis \eqref{app5}. Evidently, $\cG_-$ and $A_-$ commute.

\begin{theorem}\label{th8.4}
In the given construction $A_-=A_-(J_-)$ is a one sided GMP matrix. 
\end{theorem}

\begin{proof}
In Lemma \ref{lemapp1} it was checked that the block $B_{-1}$ has the required structure. We claim that the same is related to the block $A_{-1}$. Indeed, the function
$
\Delta(x)\tau_{-1}(x)=\Delta(x)\ell^{(-1)}_0
$
is a linear combination of $x$ and functions from the chosen  cyclic subspace, see \eqref{addm1}. That is, $\tau_{-g-2}(x)$ is of the same form.
Thus, for all $j=-1,\dots,-g-1$,  by \eqref{app0}, $x\tau_j(x)$ is a linear combination of $\{\tau_j(x)\}_{j=-1}^{-g-2}$. In other words $A_{-1}$ is of the form
\begin{equation}\label{app7}
A_{-1}=\delta_0\begin{bmatrix}
p_0^{(-1)}& \dots &
 p^{(-1)}_g
\end{bmatrix}=\delta_0 (\vp_{-1})^*,
\end{equation}
where
$$
(\vp_{-1})^*:=\int\tau_{-g-2}(x)x \begin{bmatrix}\tau_{-1}(z)& \tau_{-2}(x)&\hdots& \tau_{-g-1}(x)\end{bmatrix}d\sigma(x)
$$
$$
=\int\tau_{-g-2}(x) 
(\begin{bmatrix} x&-1&\hdots&-1
\end{bmatrix} 
+\begin{bmatrix}1&\frac{1}{\bc_g-x}&\hdots&\frac{1}{\bc_1-x}
\end{bmatrix}\hat\bC) Ld\sigma(x)
$$
$$
=\begin{bmatrix}
\ell_0^{(-1)}&\ell_0^{(-2)}&\dots&\ell_0^{(-g-1)}
\end{bmatrix}
 \int \tau_{-g-2}(x) x  d\sigma(x)=
 \frac{\begin{bmatrix}
\ell_0^{(-1)}&\ell_0^{(-2)}&\dots&\ell_0^{(-g-1)}\end{bmatrix}} {\lambda_0\ell_0^{(-1)}\ell_{g+1}^{(-g-2)}}
$$

It remains to use the commutant relation $\cG_- A_-=A_-\cG_-$.
We have
$$
\fv_{k}A_{k+1}=A_{k} \fv_{k+1},
\ \ 
\fw_k A_{k+1}+\fv_{k+1} B_{k+1}=B_{k}\fv_{k+1}+A_{k+1}\fw_{k+1}.
$$
Due to \eqref{app7},
we get
$$
A_{-2}={(\fv_{-2})_{00}}\delta_0(\vp_{-1})^* \fv_{-1}^{-1}, \ \text{that is},\ \ 
A_{-2}=\delta_0(\vp_{-2})^*, \ 
 \ (\vp_{-2})^*={(\fv_{-2})_{00}}(\vp_{-1})^* \fv_{-1}^{-1}.
$$
Generally,
\begin{equation}\label{app8}
A_{k}=\delta_0(\vp_{k})^*, \ \text{where}\ 
(\vp_{k})^*={(\fv_k)_{00}}(\vp_{k+1})^* \fv_{k+1}^{-1}.
\end{equation}

For $B$-blocks we have
$$
B_{k}=(\fv_{k+1}B_{k+1}+\fw_{k}A_{k+1}- A_{k+1}\fw_{k+1})\fv_{k+1}^{-1}.
$$
If we assume
$$
B_{k+1}=\vq_{k+1}(\vp_{k+1})^*+M_{k+1},
$$
where $M_{k+1}$ is upper-triangular, which main diagonal is $\hat\bC$, then
$$
B_k=\fv_{k+1}\vq_{k+1}(\vp_{k+1})^*\fv_{k+1}^{-1}+ \fw_k\delta_0(\vp_{k+1})^*\fv_{k+1}^{-1}+\tilde M_{k}
=(\fv_{k+1}\vq_{k+1}+ \fw_k\delta_0)(\vp_{k+1})^*\fv_{k+1}^{-1}+\tilde M_{k}
$$
where $\tilde M_{k}$ is also upper-triangular. From this relation and \eqref{app8} we get
$$
B_k=\vq_{k}(\vp_{k})^*+M_{k},
$$
where up to the first component the vector $\vq_k$ has the form
$
\frac 1{(\fv_{k})_{00}}(\fv_{k+1}\vq_{k+1}+\fw_{k}\delta_0)
$
and  by definition $M_k$ preserves the structure of its main diagonal for all $k$. Note that this is possible, since   the term $A_{k+1} \fw_{k+1}\fv_{k+1}^{-1}=\delta_0(\vp_{k+1})^* \fw_{k+1}\fv_{k+1}^{-1}$ has a non vanishing entry on the main diagonal only in the first component.

\end{proof}

\section{Acknowledgment}
I am profusely thankful to my PhD student Benjamin Eichinger for his active participation in this research (some technical results presented here were proved by him in his Master thesis and will be included in his forthcoming PhD thesis\footnote{ To be more precise, this is related to Section \ref{Section2}, and the proofs of Lemma \ref{lem:gsmpEntries}  and the second part of Theorem \ref{th73}.}, but most of all for his young energy and enthusiasm, which I borrowed a lot.  I would like to thank the Isaac Newton Institute for Mathematical Sciences, Cambridge, for their support and hospitality during the program \textit{Periodic and Ergodic Spectral Problems} (PEP), where  this paper was first presented. 
After that Barry Simon kindly suggested to organize a small informal research seminar at the Hebrew University of Jerusalem, where
in a series of lectures  a comparably complete presentation of the results was given. Thus, many thanks are due to the organizers of  PEP and especially to Barry Simon and Jonathan Breuer for their patience and  fruitful discussions during these presentations. Also, I am thankful to Alexander Kheifets and Alexander Volberg for stimulating conversations.


 \bibliographystyle{amsplain}
 \bibliography{lit2}

\providecommand{\bysame}{\leavevmode\hbox to3em{\hrulefill}\thinspace}
\providecommand{\MR}{\relax\ifhmode\unskip\space\fi MR }
\providecommand{\MRhref}[2]{%
  \href{http://www.ams.org/mathscinet-getitem?mr=#1}{#2}
}
\providecommand{\href}[2]{#2}
\begin{thebibliography}{10}

\bibitem{ALF}
L.~V. Ahlfors, \emph{{Bounded analytic functions}}, Duke. Math. J. \textbf{14}
  (1947), 1--11.

\bibitem{AKH}
N.~I. Akhiezer, \emph{{A Generalization of a Minimal Problem of
  Korkin-Zolotarev kind}}, Academic Press \textbf{4} (1936), no.~XIII.

\bibitem{Akh60}
\bysame, \emph{{Orthogonal polynomials on several intervals}}, Soviet Math.
  Dokl. (1936), 989--992.

\bibitem{AKHmp}
\bysame, \emph{{The classical moment problem and some related questions in
  analysis}}, Hafner Publishing Co., New York, 1965.

\bibitem{AKHef}
\bysame, \emph{{Elements of the Theory of Elliptic Functions}}, Amer. Math.
  Soc., Providence, 1990.

\bibitem{BD2}
Y.~M. Berezansky and M.~E. Dudkin, \emph{The strong {H}amburger moment problem
  and related direct and inverse spectral problems for block {J}acobi-{L}aurent
  matrices}, Methods Funct. Anal. Topology \textbf{16} (2010), no.~3, 203--241.

\bibitem{jreview}
J.~S. Christiansen, B.~Simon, and M.~Zinchenko, \emph{{Finite gap Jacobi
  matrices: A review.}}, Proc. Sympos. in Pure Math. \textbf{87} (2013),
  87--103.

\bibitem{CD}
M.~J. Cowen and R.~G. Douglas, \emph{{Complex geometry and operator theory}},
  Acta Math \textbf{141} (1978), 187--261.

\bibitem{KSDp}
D.~Damanik, R.~Killip, and B.~Simon, \emph{{Perturbations of orthogonal
  polynomials with periodic recursion coefficients}}, Annals of Math.
  \textbf{171} (2010), no.~3.
\bibitem{DPS}
D. Damanik, A. Pushnitski, and B. Simon, \emph{Matrix orthogonal polynomials}, Surv. Approx. Theory 4 (2008), 1-85.

\bibitem{DK}
P. A. Deift and R.~Killip, \emph{{On the absolutely continuous spectrum of
  one-dimensional Schr\"odinger operators with square summable potentials}},
  Comm. Math. Phys. \textbf{203} (1999), 341--347.

\bibitem{DUD}
M.~E. Dudkin, \emph{The inner structure of the {J}acobi-{L}aurent matrix
  related to the strong {H}amburger moment problem}, Methods Funct. Anal.
  Topology \textbf{19} (2013), no.~2, 97--107.

\bibitem{EPY}
B.~Eichinger, F.~Puchhammer, and P.~Yuditskii, \emph{{Jacobi Flow on SMP
  Matrices and Killip-Simon Problem on Two Disjoint Intervals}}, 
  Comput. Methods and Funct. Theory, DOI 10.1007/s40315-014-0104-9.

\bibitem{Fay}
J.~Fay, \emph{{Theta Functions on Riemann Surfaces}}, Lecture Notes in
  Mathematics, Springer-Verlag, 1970.

\bibitem{GNR}
F.~Gamboa, J.~Nagel, and A. Rouault, \emph{{Sum rules via large deviations}},
  arXiv: 1407.1384 (2014).

\bibitem{GZ}
L.~Golinski and A.~Zlatos, \emph{{Coefficients of orthogonal polynomials on the
  unit circle and higher-order Szeg\"o theorems}}, Constr. Approx. \textbf{26}
  (2007), no.~3.

\bibitem{Has}
M.~Hasumi, \emph{{Hardy Classes on Infinitely Connected Riemann Surfaces}},
  Lecture Notes in Math., Springer, 1983.

\bibitem{HN}
E.~Hendriksen and C.~Nijhuis, \emph{Laurent-{J}acobi matrices and the strong
  {H}amburger moment problem}, Proceedings of the {I}nternational {C}onference
  on {R}ational {A}pproximation, {ICRA}99 ({A}ntwerp), vol.~61, 2000,
  pp.~119--132.

\bibitem{JN}
W.~B. Jones and O.~Nj{\aa}stad, \emph{{Orthogonal Laurent polynomials and
  strong moment theory: a survey. Continued fractions and geometric function
  theory}}, J. Comput. Appl. Math. \textbf{105} (1999), no.~1-2.

\bibitem{KS}
R.~Killip and B.~Simon, \emph{{Sum rules for Jacobi matrices and their
  applications to spectral theory}}, Annals of Math. \textbf{158} (2003),
  no.~2.

\bibitem{K2004}
S.~Kupin, \emph{{On a spectral property of Jacobi matrices}}, Proc. Amer. Math.
  Soc. \textbf{132} (2004), no.~5.

\bibitem{LNS}
A.~Laptev, S.~Naboko, and O.~Safronov, \emph{{On new relations between spectral
  properties of Jacobi matrices and their coefficients}}, Comm. Math. Phys.
  \textbf{241} (2003), no.~1.

\bibitem{LU}
M.~Lukic, \emph{{On a conjecture for higher-order Szeg\"o theorems}}, Constr.
  Approx. \textbf{38} (2013), 161--169.

\bibitem{MaT}
V.~Matveev, \emph{{30 years of finite-gap integration theory}}, Phil. Trans. R.
  Soc. A \textbf{366} (2008).

\bibitem{MUM}
D.~Mumford, \emph{{Tata lectures on theta, vol. I, II.}}, MA: Birkh\"auser,
  Boston, 1983.

\bibitem{NPVY}
F.~Nazarov, F.~Peherstorfer, A.~Volberg, and P.~Yuditskii, \emph{{On
  generalized sum rules for Jacobi matrices}}, Int. Math. Res. Not. (2005),
  no.~3, 155--186.

\bibitem{PVY}
F.~Peherstorfer, A.~Volberg, and P.~Yuditskii, \emph{{CMV matrices with
  asymptotically constant coefficients. Szego-Blaschke class, scattering
  theory}}, Journal of Functional Analysis \textbf{256} (2009), 2157--2210.

\bibitem{PY}
F.~Peherstorfer and P.~Yuditskii, \emph{{Asymptotic behaviour of polynomials
  orthonormal on a homogeneous set}}, J. Anal. Math. \textbf{89} (2003),
  113--154.

\bibitem{Pom}
Ch. Pommerenke, \emph{{On the Green's function of Fuchsian groups}}, Ann. Acad.
  Sci. Fenn. \textbf{2} (1976), 409--427.

\bibitem{Pom0}
Ch. Pommerenke, \emph{Uber die analytische kapazit\"at}, Archiv der
  Mathematik \textbf{11} (1960), no.~1, 270--277 (German).

\bibitem{POT}
V. P. Potapov, \emph{{The Multiplicative Structure of J-contractive Matrix
  Functions}}, American Mathematical Society translations, American
  Mathematical Society, 1960.

\bibitem{REMA11}
C.~Remling, \emph{{The absolutely continuous spectrum of Jacobi matrices}},
  Annals of Math. \textbf{174} (2011), no.~2, 125--171.

\bibitem{2005v1}
B.~Simon, \emph{{Orthogonal polynomials on the unit circle. Part 1. Classical
  theory }}, American Mathematical Society Colloquium Publications, American
  Mathematical Society, Providence, 2005.

\bibitem{2005v2}
\bysame, \emph{{Orthogonal polynomials on the unit circle. Part 2. Spectral
  theory}}, American Mathematical Society Colloquium Publications, American
  Mathematical Society, Providence, 2005.

\bibitem{BS}
\bysame, \emph{{Szeg\"o's Theorem and Its Descendants: Spectral Theory for
  $L^2$ Perturbations of Orthogonal Polynomials}}, Princeton University Press,
  New Jersey, 2011.

\bibitem{SZ}
B.~Simon and A.~Zlatos, \emph{{Higher-order Szeg\"o theorems with two singular
  points}}, J. Approx. Theory \textbf{134} (2005), no.~1, 114--129.

\bibitem{Sim1}
K.~Simonov, \emph{{Orthogonal Matrix Laurent Polynomials}}, Mathematical Notes
  \textbf{79} (2006), no.~2, 291--295.

\bibitem{Sim2}
\bysame, \emph{{Strong matrix moment problem of Hamburger}, methods of
  functional analysis and topology}, Mathematical Notes \textbf{12} (2006),
  no.~2, 183--196.

\bibitem{SY}
M.~Sodin and P.~Yuditskii, \emph{{Almost periodic Jacobi matrices with
  homogeneous spectrum, infinite-dimensional Jacobi inversion, and Hardy spaces
  of character-automorphic functions}}, J. Geom. Anal. \textbf{7} (1997),
  387--435.

\bibitem{GT}
G.~Teschl, \emph{{Jacobi Operators and Completely Integrable Nonlinear
  Lattices}}, Mathematical surveys and monographs, vol.~72, American
  Mathematical Society, Providence.

\bibitem{VY}
A.~Volberg and P.~Yuditskii, \emph{{On the inverse scattering problem for
  Jacobi matrices with the spectrum on an interval, a finite system of
  intervals or a Cantor set of positive length}}, Commun. Math. Phys.
  \textbf{226} (2002), no.~3, 567--605.

\bibitem{vN}
J.~von Neumann, \emph{{Charakterisierung des Spektrums eines
  Integraloperators}}, Actualit\'es Sci. Indust. \textbf{229} (1935).

\bibitem{WID}
H.~Widom, \emph{{$H_p$ sections of vector bundles over Riemann surfaces}}, Ann.
  Math. \textbf{94} (1971), 304--324.

\end{thebibliography}

\smallskip

\noindent
{Institute for Analysis, Johannes Kepler University Linz,
A-4040 Linz, Austria} 

\noindent\textit{E-mail address}:
{petro.yudytskiy@jku.at,}
\end{document}